\documentclass{amsart}

\usepackage{amsmath,amssymb,amscd,amsthm,latexsym,accents,stmaryrd}
\usepackage{amsfonts}
\usepackage[latin1]{inputenc}
\usepackage{graphicx}
\usepackage{float}
\usepackage{subfigure}
\usepackage[all]{xy}
\usepackage{enumerate}
\usepackage{amsaddr}
\usepackage{tikz}
\usetikzlibrary{automata,positioning}
\newtheorem{theor}{Theorem}[subsection]
\newtheorem{lem}[theor]{Lemma}
\newtheorem{prop}[theor]{Proposition}
\newtheorem{cor}[theor]{Corollary}
\theoremstyle{definition}
\newtheorem{defin}[theor]{Definition}
\theoremstyle{remark}
\newtheorem*{rem}{Remark}
\newtheorem*{rems}{Remarks}

\newtheorem*{exs}{Examples}

\newcommand{\dem}{\noindent \emph{Proof. }}
\newcommand{\findem}{\hfill $\Box$}

\newdir{ >}{{}*!/-7pt/\dir{>}}

\def\R{{\mathbb{R}}}

\def\im{{\rm{im}\,}}

\def\length{{\rm{length}}}
\def\A{{\mathcal{A}}}
\def\B{{\mathcal{B}}}

\def\deg{{\rm{deg}}}

\begin{document}

\title{The homology graph of a higher dimensional automaton}

\author{Thomas Kahl}

\address{Centro de Matem\'atica,
Universidade do Minho, Campus de Gualtar, \\
4710-057 Braga,
Portugal
}

\email{kahl@math.uminho.pt%, \textit{Phone}\rm{:} +351 253 604 357
}

\thanks{This research has been supported by FEDER funds through ``Programa Operacional Factores de Competitividade - COMPETE'' and by FCT -  \emph{Fundação para a Ciência e a Tecnologia} through projects Est-C/MAT/UI0013/2011 and PTDC/MAT/0938317/2008.}



%\date{\today . \\  \indent \emph{Draft}: \svnrev}

\begin{abstract}
Higher dimensional automata, i.e. labelled precubical sets, model concurrent systems. We introduce the homology graph of an HDA, which is a directed graph whose nodes are the homology classes of the HDA. We show that the homology graph is invariant under homeomorphic abstraction, i.e. under weak morphisms that are homeomorphisms.
\end{abstract}

\maketitle

\section*{Introduction}

Higher dimensional automata constitute a powerful model for concurrent systems  \cite{FajstrupGR, vanGlabbeek, Goubault}. A higher dimensional automaton (HDA) is a precubical set (i.e. a cubical set without degeneracies) with initial and final states and labels in a monoid. The labelled edges of an HDA represent the actions of the system modelled by the HDA. Squares and higher dimensional cubes indicate independence of actions: if two actions $a$ and $b$ are enabled in a state and are independent in the sense that they may be executed in any order or even simultaneously without any observable difference, then this is indicated, as in figure \ref{fig1}, by a square linking the two execution sequences $ab$ and $ba$. Similarly, the independence of $n$ actions is represented by an $n$-cube. 

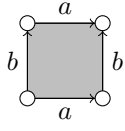
\begin{figure}[t]
\begin{tikzpicture}[initial text={},on grid]

\path[draw, fill=lightgray] (0,0)--(1,0)--(1,1)--(0,1)--cycle;

 \node[state,minimum size=0pt,inner sep =2pt,fill=white] (q_0)   {}; 
    
   \node[state,minimum size=0pt,inner sep =2pt,fill=white] (q_2) [right=of q_0,xshift=0cm] {};
   
   \node[state,minimum size=0pt,inner sep =2pt,fill=white] [above=of q_0, yshift=0cm] (q_3)   {};

   \node[state,minimum size=0pt,inner sep =2pt,fill=white] (q_5) [right=of q_3,xshift=0cm] {}; 
   
    \path[->] 
    (q_0) edge[below] node {$a$} (q_2)
    (q_3) edge[above]  node {$a$} (q_5)
    (q_0) edge[left]  node {$b$} (q_3)
    (q_2) edge[right]  node {$b$} (q_5);

\end{tikzpicture}
\caption{Cubes represent independence of actions} \label{fig1}
\end{figure}

The 1-skeleton of a precubical set is directed (multi)graph, and therefore an HDA can be called a ``directed topological object". Various notions of directed homology have been defined in the literature for HDAs and other directed topological objects \cite{FahrenbergDiH, GaucherHomol, GoubaultJensen, GrandisBook, Husainov}. In this paper, we take a still different approach to directed homology and introduce the homology graph of an HDA. This is a directed graph whose nodes are the homology classes of the HDA. The definition and some basic properties of the homology graph of an HDA are given in section \ref{hograph}.

Our main result on the homology graph is that it is invariant under homeomorphic abstraction. Homeomorphic abstraction is a preorder relation for HDAs that can be seen as a kind of T-homotopy equivalence in the sense of Gaucher and Goubault \cite{GaucherGoubault}. A homeomorphic abstraction of an HDA provides a smaller representation of the modelled system. The definition of homeomorphic abstraction is based on the concept of weak morphism that has been introduced in \cite{weakmor}. Roughly speaking, a weak morphism of precubical sets is a continuous map between the geometric realisations that sends vertices to vertices and subdivided cubes to subdivided cubes. The precise definition is given in section \ref{wm}. An HDA $\A$ is said to be a homeomorphic abstraction of an HDA $\B$ if there exists a weak morphism from $\A$ to $\B$ that preserves labels of paths, that is a bijection on initial and on final states and that is a homeomorphism. In section \ref{homeoabs}, we show that the homology graph is invariant under weak morphisms that are homeomorphisms and thus under homeomorphic abstraction.

\section{Preliminaries on precubical sets and HDAs}

This section, which is taken from \cite{weakmor}, contains some basic and well-known material on precubical sets and higher dimensional automata.

\subsection{Precubical sets}

A \emph{precubical set} is a graded set $P = (P_n)_{n \geq 0}$ with  \emph{boundary operators} $d^k_i: P_n \to P_{n-1}$ $(n>0,\;k= 0,1,\; i = 1, \dots, n)$ satisfying the relations $d^k_i\circ d^l_{j}= d^l_{j-1}\circ d^k_i$ $(k,l = 0,1,\; i<j)$ \cite{FahrenbergThesis, %Fajstrup, 
FajstrupGR, GaucherGoubault, Goubault}. The least $n\geq 0$ such that $P_i = \emptyset$ for all $i>n$ is called the \emph{dimension} of $P$. If no such $n$ exists, then the dimension of $P$ is $\infty$. If $x\in P_n$, we say that $x$ is of \emph{degree} $n$ and write $\deg(x) = n$. The elements of degree $n$ are called the \emph{$n$-cubes} of $P$. The elements of degree $0$ are also called the \emph{vertices} or the \emph{nodes} of $P$. A morphism of precubical sets is a morphism of graded sets that is compatible with the boundary operators. 

The category of precubical sets can be seen as the presheaf category of functors $\Box^{\mbox{\tiny op}} \to \bf Set$ where $\Box$ is the small subcategory of the category of topological spaces whose objects are the standard $n$-cubes $[0,1]^n$ $(n \geq 0)$ and whose non-identity morphisms are composites of the maps $\delta^k_i\colon [0,1]^n\to [0,1]^{n+1}$ ($k \in \{0,1\}$, $n \geq 0$, $i \in  \{1, \dots, n+1\}$) given by $\delta_i^k(u_1,\dots, u_n)= (u_1,\dots, u_{i-1},k,u_i \dots, u_n)$. Here, we use the convention that given a topological space $X$, $X^0$ denotes the one-point space $\{()\}$.

%Here, we consider the $0$-cube as the one-point space $[0,1]^0 = \{()\}$. %The \emph{precubical $n$-cube} is the $n$-dimensional precubical set $\llbracket 0,1\rrbracket^n = \Box(-,[0,1]^n)$. By Yoneda's Lemma, an element $x$ of degree $n$ of a precubical set $P$ determines a unique morphism of precubical sets $x_{\sharp}\colon {\llbracket 0,1\rrbracket}^n \to P$ such that $x_{\sharp}(id_{[0,1]^n}) = x$.

%\subsection{Normal form of an iterated boundary}
%Let $P$ be a precubical set and $x \in P_n$ $(n>1)$. Consider an element of the form $d_{j_1}^{k_1} \cdots d_{j_r}^{k_r}x$ where $2\leq r \leq n$, $k_1,\dots, k_r \in \{0,1\}$ and $j_s \in \{1, \dots, n-r+s\}$ $(s = 1, \dots, r)$. Then there exist integers $i_1 < \dots < i_r$ and a permutation $\sigma$ of the set $\{1, \dots, r\}$ such that $i_s \in \{1, \dots, n-r+s\}$ $(s = 1, \dots, r)$ and  $d_{i_1}^{k_{\sigma(1)}} \cdots d_{i_r}^{k_{\sigma(r)}}x =  d_{j_1}^{k_1} \cdots d_{j_r}^{k_r}x$. 

\subsection{Precubical subsets}

A \emph{precubical subset} of a precubical set $P$ is a graded subset of $P$ that is stable under the boundary  operators. It is clear that a precubical subset is itself a precubical set. Note that unions and intersections of precubical subsets are precubical subsets and that images and preimages of precubical subsets under a morphism  of precubical sets are precubical subsets. %An $M$-HDA $\A = (P,I,F,\lambda)$ is a \emph{subautomaton} of an $M$-HDA $\A' = (P'.I',F', \lambda')$ if $P$ is a precubical subset of $P'$, $I \subseteq I'$, $F\subseteq F'$ and $\lambda = \lambda'|_{P_1}$.

\subsection{Intervals}

Let $k$ and $l$ be integers such that $k \leq l$. The \emph{precubical interval}  $\llbracket k,l \rrbracket$ is the at most $1$-dimensional precubical set defined by $\llbracket k,l \rrbracket_0 = \{k,\dots , l\}$, $\llbracket k,l \rrbracket_1 = \{[k,k+1], \dots , [l-1,l]\}$, $d_1^0[j-1,j] = j-1$ and $d_1^1[j-1,j] = j$. We shall use the abbreviations $\rrbracket k,l \llbracket = \llbracket k,l\rrbracket \setminus \{k,l\}$, $\llbracket k,l \llbracket = \llbracket k,l\rrbracket \setminus \{l\}$ and $\rrbracket k,l \rrbracket = \llbracket k,l\rrbracket \setminus \{k\}$.

\subsection{Tensor product} 

Given two graded sets $P$ and $Q$, the \emph{tensor product} $P\otimes Q$ is the graded set defined by $(P\otimes Q)_n = \coprod \limits_{p+q = n} P_p\times Q_q$. If $P$ and $Q$ are precubical sets, then $P\otimes Q$ is a precubical set with respect to the boundary operators given by
$$d_i^k(x,y) = \left\{ \begin{array}{ll} (d_i^kx,y), & 1\leq i \leq \deg(x),\\
(x,d_{i-\deg(x)}^ky), & \deg(x)+1 \leq i \leq \deg(x) + \deg(y)
\end{array}\right.$$ 
(cf. \cite{FahrenbergThesis}). 
%The tensor product of an $M$-HDA $\A = (P,I,F,\lambda)$ and an $M$-HDA $\B= (Q,J,G,\mu)$ is the $M$-HDA $$\A\otimes \B = (P\otimes Q,I\times J,F\times G, (P_1\times Q_0) \amalg (P_0\times Q_1) \to P_1\amalg Q_1 \stackrel{(\lambda,\mu)}{\longrightarrow} M).$$ 
The tensor product turns the categories of graded and precubical sets %and $M$-HDAs 
into monoidal categories.  

The $n$-fold tensor product of a graded or precubical set $P$ is denoted by $P^{\otimes n}$. Here, we use the convention $P^{\otimes 0} = \llbracket 0,0\rrbracket = \{0\}$. The \emph{precubical $n$-cube} is the precubical set $\llbracket 0,1\rrbracket^ {\otimes n}$. The only element of degree $n$ in $\llbracket 0,1\rrbracket^ {\otimes n}$ will be denoted by $\iota_n$. We thus have $\iota_0 = 0$ and $\iota_n = (\underbrace{ [0,1] ,\dots , [ 0,1]}_{n\; { \rm{times}}})$ for $n>0$.

\subsection{The morphism corresponding to an element} 

Let $x$ be an element of degree $n$ of a precubical set $P$. Then there exists a unique morphism of precubical sets $x_{\sharp}\colon \llbracket 0,1\rrbracket ^{\otimes n}\to P$ such that $x_{\sharp}(\iota_n) = x$. Indeed, 
by the Yoneda lemma, there exist unique morphisms of precubical sets $f\colon \Box(-,[0,1]^n) \to P$ and $g\colon \Box(-,[0,1]^n) \to \llbracket 0,1\rrbracket ^{\otimes n}$ such that $f(id_{[0,1]^n}) = x$ and $g(id_{[0,1]^n}) = \iota_n$. The map $g$ is an isomorphism, and $x_{\sharp} = f\circ g^{-1}$. %In order to show that $g$ is bijective, note first that this is clear in degrees $\geq n$. Consider $0 \leq m < n$. For an element $$y = (y_1, \dots,y_n) \in \llbracket 0,1\rrbracket^{\otimes n}_m  = \coprod \limits_{\sum \limits _{s=0}^n p_s = m} \llbracket 0,1\rrbracket_{p_1} \times \cdots \times \llbracket 0,1\rrbracket_{p_n}$$ there exist unique integers $1 \leq i_1 < \dots < i_{n-m}\leq n$ and $k_1 \dots, k_{n-m} \in \{0,1\}$ such that $y_{i_j} = k_j$ for $j = 1, \dots, n-m$ and $y_i = [0,1]$ for $i \notin \{i_1, \dots, i_{n-m}\}$. One then has $y = d_{i_1}^{k_1}\cdots d_{i_{n-m}}^{k_{n-m}}\iota_n $. The inverse of the map $g$ in degree $m$ is given by $y \mapsto d_{i_1}^{k_1}\cdots d_{i_{n-m}}^{k_{n-m}}id_{[0,1]^n}$.

\subsection{Paths} 

A \emph{path of length $k$} in a precubical set $P$ is a morphism of precubical sets $\omega \colon \llbracket 0,k \rrbracket \to P$. The set of paths in $P$ is denoted by $P^{\mathbb I}$. If $\omega \in P^{\mathbb I}$ is a path of length $k$, we write $\length(\omega) = k$. The \emph{concatenation} of two paths $\omega \colon \llbracket 0,k \rrbracket \to P$ and $\nu \colon \llbracket 0,l \rrbracket \to P$ with $\omega (k) = \nu (0)$ is the path $\omega \cdot \nu\colon \llbracket 0,{k+l} \rrbracket \to P$ defined by 
$$\omega \cdot \nu (j) = \left\{\begin{array}{ll} \omega (j), & 0\leq j \leq k,\\ \nu(j-k), & k \leq j \leq k+l \end{array}\right.$$ and $$\omega \cdot \nu ([j-1,j]) = \left\{\begin{array}{ll} \omega ([j-1,j]) & 0< j \leq k,\\ \nu([j-k-1,j-k]) & k < j \leq k+l. \end{array}\right. $$ Clearly, concatenation is associative. Note that for any path $\omega \in P^{\mathbb I}$ of length $k \geq 1$ there exists a unique sequence $(x_1, \dots , x_k)$ of elements of $P_1$ such that $d_1^0x_{j+1} = d_1^1x_j$ for all $1\leq j < k$ and $\omega = x_{1\sharp} \cdots x_{k\sharp}$.

\subsection{Geometric realisation}
 \label{geomreal}  

The \emph{geometric realisation} of a precubical set $P$ is the quotient space $|P|=(\coprod _{n \geq 0} P_n \times [0,1]^n)/\sim$ where the sets $P_n$ are given the discrete topology and the equivalence relation is given by
$$(d^k_ix,u) \sim (x,\delta_i^k(u)), \quad  x \in P_{n+1},\; u\in [0,1]^n,\; i \in  \{1, \dots, n+1\},\; k \in \{0,1\}$$ (see \cite{FahrenbergThesis, %Fajstrup, 
FajstrupGR,  GaucherGoubault, Goubault}). The geometric realisation of a morphism of precubical sets $f\colon P \to Q$ is the continuous map $|f|\colon |P| \to |Q|$ given by $|f|([x,u])= [f(x),u]$. We remark that the geometric realisation is a functor from the category of precubical sets to the category $\bf Top$  of topological spaces. The geometric realisation is left adjoint to the \emph{singular precubical set} functor $S$ defined by $S(X)_n = {\bf Top}([0,1]^n,X)$, $d_i^k\sigma = \sigma\circ \delta_i^ k$ and $S(f)(\sigma) =f\circ \sigma$.

\begin{exs}
(i) The geometric realisation of the precubical $n$-cube can be identified with the standard $n$-cube by means of the homeomorphism $[0,1]^n \to |\llbracket 0,1\rrbracket^ {\otimes n}|$, $u \mapsto [\iota_n,u]$.

(ii) The geometric realisation of the precubical interval $\llbracket k,l \rrbracket$ can be identified with the closed interval $[k,l]$ by means of the homeomorphism $|\llbracket k,l \rrbracket| \to [k,l]$ given by $[j,()] \mapsto j$ and $[[j-1,j],t] \mapsto j-1+t$. Using this correspondence, the geometric realisation of a precubical path $\llbracket 0, k\rrbracket \to P$ can be seen as a path $[0,k] \to |P|$, and under this identification we have that $|\omega \cdot \nu| = |\omega|\cdot |\nu|$. 
\end{exs}

We note that for every element $a \in |P|$ there exist a unique integer $n\geq 0$, a unique element $x \in P_n$ and a unique element $u \in ]0,1[^ n$ such that $a = [x,u]$.

The geometric realisation of a precubical set $P$ is a CW-complex \cite{GaucherGoubault}. The $n$-skeleton of $|P|$ is the geometric realisation of the precubical subset $P_{\leq n}$ of $P$ defined by $(P_{\leq n})_m = P_m$ $(m\leq n)$. The closed $n$-cells of $|P|$ are the spaces $|x_{\sharp}(\llbracket 0,1\rrbracket^{\otimes n})|$ where $x \in P_n$. The characteristic map of the cell $|x_{\sharp}(\llbracket 0,1\rrbracket^{\otimes n})|$ is the map $[0,1]^n \stackrel{\approx}{\to} |\llbracket 0,1\rrbracket ^{\otimes n}| \stackrel{|x_{\sharp}|}{\to} |P|, u \mapsto [x,u]$. The geometric realisation of a precubical subset $Q$ of $P$ is a subcomplex of $|P|$.

The geometric realisation is a comonoidal functor with respect to the natural continuous map $\psi_{P,Q}\colon |P\otimes Q| \to |P| \times |Q|$ given by \begin{eqnarray*}\lefteqn{\psi_{P,Q} ([(x,y),(u_1,\dots ,u_{\deg(x)+\deg(y)})])}\\ &=& ([x,(u_1,\dots , u_{\deg(x)})],[y,(u_{\deg(x)+1}, \dots u_{\deg(x)+\deg(y)}]).\end{eqnarray*} If $P$ and $Q$ are finite, then $\psi_{P,Q}$ is a homeomorphism and permits us to identify $|P\otimes Q|$ with $|P|\times |Q|$. We may thus identify the geometric realisation of a precubical set of the form $\llbracket k_1,l_1\rrbracket \otimes \cdots \otimes \llbracket k_n,l_n\rrbracket$ $(k_i<l_i)$ with the product $[k_1,l_1]\times \cdots \times [k_n,l_n]$ by means of the correspondence $$\left[([i_1,i_1+1], \dots,[i_n,i_n+1]),(u_1,\dots,u_n)\right] \mapsto (i_1+u_1,\dots ,i_n+u_n).$$

\subsection{Higher dimensional automata} \label{HDAdef}

Let $M$ be a monoid. A \emph{higher dimensional automaton over} $M$ (abbreviated $M$-HDA or simply HDA) is a tuple $\A = (P,I,F, \lambda)$ where $P$ is a precubical set, $I \subseteq P_0$ is a set of \emph{initial states}, $F \subseteq P_0$ is a set of \emph{final states}, and $\lambda \colon P_1 \to M$ is a map, called the \emph{labelling function}, such that $\lambda (d_i^0x) = \lambda (d_i^1x)$ for all $x \in P_2$ and $i \in \{1,2\}$. A \emph{morphism} from an $M$-HDA $\A = (P,I, F,\lambda)$ to an $M$-HDA $\B = (P',I', F',\lambda')$ is a morphism of precubical sets  $f\colon P \to P'$  such that $f(I) \subseteq I'$, $f(F) \subseteq F'$ and  $\lambda'(f(x)) = \lambda(x)$ for all $x \in P_1$.

This definition of higher dimensional automata is essentially the same as the one in \cite{vanGlabbeek}. Besides the fact that we consider a monoid and not just a set of labels, the only difference is that in \cite{vanGlabbeek} an HDA is supposed to have exactly one initial state. Note that 1-dimensional $M$-HDAs and morphisms of 1-dimensional $M$-HDAs are the same as automata over $M$ and automata morphisms as defined in \cite{Sakarovitch}.

\section{Weak morphisms} 
\label{wm}

Morphisms of HDAs are often too rigid to be useful for the comparison of HDAs. For instance, the two HDAs $\A$ and $\B$ in figure \ref{fig2} model the same system, but there are no morphisms of HDAs between them. In order to overcome this problem, weak morphisms have been introduced in \cite{weakmor}. In figure \ref{fig2}, there exists a weak morphism from $\A$ to $\B$. This section contains the definition and the basic properties of weak morphisms.  
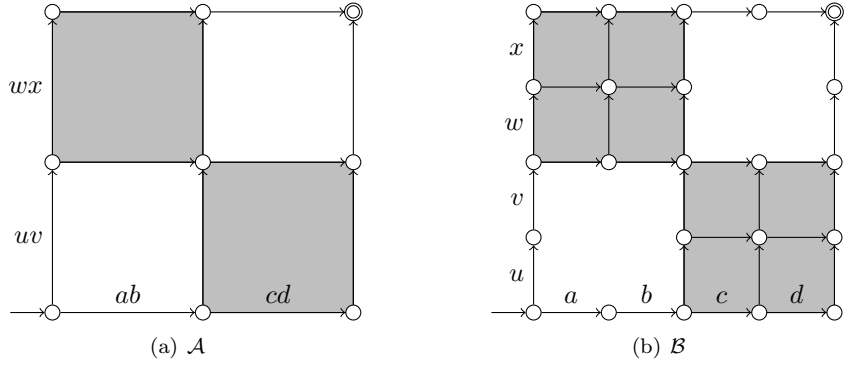
\begin{figure}
\subfigure[$\A$]
{
\begin{tikzpicture}[initial text={},on grid]

\path[draw, fill=lightgray] (0,2)--(2,2)--(2,4)--(0,4)--cycle
(2,0)--(4,0)--(4,2)--(2,2)--cycle;

 \node[state,minimum size=0pt,inner sep =2pt,initial,fill=white] (q_0)   {};

\node[state,minimum size=0pt,inner sep =2pt,fill=white] (q_2) [right=of q_0,xshift=1cm] {};

      \node[state,minimum size=0pt,inner sep =2pt,fill=white] (q_6) [above=of q_0,yshift=1cm] {};

      \node[state,minimum size=0pt,inner sep =2pt,fill=white] (q_8) [above=of q_2,yshift=1cm] {};

      \node[state,minimum size=0pt,inner sep =2pt,fill=white] (q_11) [above=of q_6,yshift=1cm] {};

      \node[state,minimum size=0pt,inner sep =2pt,fill=white] (q_13) [above=of q_8,yshift=1cm] {};

   \node[state,minimum size=0pt,inner sep =2pt,fill=white] (p_1) [right=of q_2,xshift=1cm] {};

      \node[state,minimum size=0pt,inner sep =2pt,fill=white] (p_7) [above=of p_1,yshift=1cm] {};

      \node[state,minimum size=0pt,accepting,inner sep =2pt,fill=white] (p_12) [right=of q_13,xshift=1cm] {};

    \path[->] 
    (q_0) edge[above]  node {$ab$} (q_2)
    (q_2) edge[above]  node {$cd$} (p_1)
    
	(q_0) edge[left]  node {$uv$} (q_6)
    (q_2) edge[right]  node {} (q_8)
    (p_1) edge[left]  node {} (p_7)
    
    (q_6) edge[above]  node {} (q_8)
    (q_8) edge[above]  node {} (p_7)
    
    (q_6) edge[left]  node {$wx$} (q_11)
	(q_8) edge[right]  node {} (q_13)
    (p_7) edge[left]  node {} (p_12)
   
    (q_11) edge[above]  node {} (q_13)
    (q_13) edge[above]  node {} (p_12)
    ;
\end{tikzpicture}
} 
\hspace{1cm}
\subfigure[$\B$] 
{
\begin{tikzpicture}[initial text={},on grid]

\path[draw, fill=lightgray] (0,2)--(2,2)--(2,4)--(0,4)--cycle
(2,0)--(4,0)--(4,2)--(2,2)--cycle;

 \node[state,minimum size=0pt,inner sep =2pt,initial,fill=white] (q_0)   {}; 
    
   \node[state,minimum size=0pt,inner sep =2pt,fill=white] (q_1) [right=of q_0,xshift=0cm] {};

\node[state,minimum size=0pt,inner sep =2pt,fill=white] (q_2) [right=of q_1,xshift=0cm] {};   
   
   \node[state,minimum size=0pt,inner sep =2pt,fill=white] [above=of q_0, yshift=0cm] (q_3)   {};

   \node[state,minimum size=0pt,inner sep =2pt,fill=white] (q_5) [above=of q_2,xshift=0cm] {}; 
   
      \node[state,minimum size=0pt,inner sep =2pt,fill=white] (q_6) [above=of q_3,yshift=0cm] {};
      
      \node[state,minimum size=0pt,inner sep =2pt,fill=white] (q_7) [right=of q_6,yshift=0cm] {};

      \node[state,minimum size=0pt,inner sep =2pt,fill=white] (q_4) [above=of q_7,yshift=0cm] {};

      \node[state,minimum size=0pt,inner sep =2pt,fill=white] (q_8) [above=of q_5,yshift=0cm] {};
      
      \node[state,minimum size=0pt,inner sep =2pt,fill=white] (q_9) [above=of q_6,yshift=0cm] {};
      
      \node[state,minimum size=0pt,inner sep =2pt,fill=white] (q_10) [above=of q_8,yshift=0cm] {};
      
      \node[state,minimum size=0pt,inner sep =2pt,fill=white] (q_11) [above=of q_9,yshift=0cm] {};
      
      \node[state,minimum size=0pt,inner sep =2pt,fill=white] (q_12) [right=of q_11,xshift=0cm] {};
      
      \node[state,minimum size=0pt,inner sep =2pt,fill=white] (q_13) [above=of q_10,yshift=0cm] {};

 \node[state,minimum size=0pt,inner sep =2pt,fill=white] (p_0) [right=of q_2,xshift=0cm]  {}; 
    
   \node[state,minimum size=0pt,inner sep =2pt,fill=white] (p_1) [right=of p_0,xshift=0cm] {};

   \node[state,minimum size=0pt,inner sep =2pt,fill=white] (p_3) [above=of p_0,xshift=0cm] {};

   \node[state,minimum size=0pt,inner sep =2pt,fill=white] (p_4) [above=of p_1,xshift=0cm] {};

      \node[state,minimum size=0pt,inner sep =2pt,fill=white] (p_6) [above=of p_0,yshift=1cm] {};
      
      \node[state,minimum size=0pt,inner sep =2pt,fill=white] (p_7) [above=of p_4,yshift=0cm] {};
      
      \node[state,minimum size=0pt,inner sep =2pt,fill=white] [above=of p_7, yshift=0cm] (p_8)   {};

      \node[state,minimum size=0pt,inner sep =2pt,fill=white] (p_11) [right=of q_13,yshift=0cm] {};
      
      \node[state,minimum size=0pt,accepting,inner sep =2pt,fill=white] (p_12) [right=of p_11,xshift=0cm] {};

    \path[->] 
    (q_0) edge[above]  node {$a$} (q_1)
    (q_1) edge[above]  node {$b$} (q_2)
	(q_2) edge[above]  node {$c$} (p_0)    
    (p_0) edge[above]  node {$d$} (p_1)
    
	(q_0) edge[left]  node {$u$} (q_3)
    (q_2) edge[right]  node {} (q_5)
    (p_0) edge[right]  node {} (p_3)
    (p_1) edge[left]  node {} (p_4)
    
    (q_5) edge[above]  node {} (p_3)
    (p_3) edge[above]  node {} (p_4)
    
    (q_3) edge[left]  node {$v$} (q_6)
    (q_5) edge[left]  node {} (q_8)    
    (p_3) edge[left]  node {} (p_6)	
    (p_4) edge[left]  node {} (p_7)
    
    (q_6) edge[above]  node {} (q_7)
    (q_7) edge[above]  node {} (q_8)
    (q_8) edge[above]  node {} (p_6)
    (p_6) edge[above]  node {} (p_7)
    
    (q_6) edge[left]  node {$w$} (q_9)
	(q_7) edge[left]  node {} (q_4)    
    (q_8) edge[right]  node {} (q_10)
    (p_7) edge[left]  node {} (p_8)

    (q_9) edge[above]  node {} (q_4)
    (q_4) edge[above]  node {} (q_10)
    
    (q_9) edge[left]  node {$x$} (q_11)
	(q_4) edge[left]  node {} (q_12)   
    (q_10) edge[right]  node {} (q_13)
    (p_8) edge[left]  node {} (p_12)
    
    (q_11) edge[above]  node {} (q_12)
    (q_12) edge[above]  node {} (q_13)
    (q_13) edge[above]  node {} (p_11)
    (p_11) edge[above]  node {} (p_12)
    ;
\end{tikzpicture}
}
\caption{Two HDAs modelling the same system. Parallel edges are meant to have the same label.}
\label{fig2}
\end{figure}

\subsection{Weak morphisms of precubical sets} 
A \emph{weak morphism} from a precubical set $P$ to a precubical set $Q$ is a continuous map $f\colon |P| \to |Q|$ such that the following two conditions hold:
\begin{enumerate}
\item for every vertex $v\in P_0$ there exists a (necessarily unique) vertex $f_0(v)\in Q_0$ such that $f([v,()]) = [f_0(v),()]$;
\item for all integers $n, k_1, \dots, k_n\geq 1$ and every morphism of precubical sets $\xi \colon \llbracket 0,{k_1} \rrbracket\otimes \cdots \otimes \llbracket 0,{k_n} \rrbracket \to P$ there exist integers $l_1, \dots, l_n\geq 1$, a morphism of precubical sets $\chi \colon \llbracket 0,{l_1} \rrbracket\otimes \cdots \otimes \llbracket 0,{l_n} \rrbracket \to Q$ and a homeo\-morphism 
\begin{eqnarray*}\lefteqn{\phi\colon  |\llbracket 0,{k_1} \rrbracket\otimes \cdots \otimes \llbracket 0,{k_n} \rrbracket| = [0,k_1] \times \cdots \times [0,k_n]}\\ &\to& |\llbracket 0,{l_1} \rrbracket\otimes \cdots \otimes \llbracket 0,{l_n} \rrbracket|= [0,l_1] \times \cdots \times [0,l_n] \end{eqnarray*} 
such that $f\circ |\xi| = |\chi|\circ \phi$ and $\phi$ is a dihomeomorphism, i.e. $\phi$ and $\phi^{-1}$ preserve the natural partial order of $\R^n$. 
\end{enumerate}

It is clear that the geometric realisation of a morphism of precubical sets is a weak morphism. We remark that weak morphisms are stable under composition. It is also important to note that the integers $l_1, \dots, l_n\geq 1$, the morphism of precubical sets $\chi$ and the dihomeomorphism $\phi$ in condition (2) above are unique and that $\phi$ is itself a weak morphism \cite[2.3.5]{weakmor}. In the case of the morphism of precubical sets $x _{\sharp} \colon \llbracket 0,1\rrbracket ^ {\otimes n}\to P$ corresponding to an element $x \in P_n$ $(n>0)$, we shall use the notation $R_x = \llbracket 0,{l_1} \rrbracket\otimes \cdots \otimes \llbracket 0,{l_n} \rrbracket$, $\phi_x = \phi$ and $x_{\flat} = \chi$. We shall also write $\dot{R}_x = \rrbracket 0,{l_1} \llbracket\otimes \cdots \otimes \rrbracket 0,{l_n} \llbracket$ and $\partial R_x = R_x \setminus \dot{R}_x$.

\subsection{Weak morphisms and paths} 
Let $f \colon |P| \to |Q|$ be a weak morphism of precubical sets and $\omega \colon \llbracket 0,k \rrbracket \to P$ $(k \geq 0)$ be a path. If $k > 0$, we denote by $f^{\mathbb I}(\omega)$ the unique path $\nu \colon \llbracket 0,{l} \rrbracket \to Q$ for which there exists a dihomeomorphism $\phi \colon |\llbracket 0,k \rrbracket| = [0,k] \to |\llbracket 0,{l} \rrbracket| = [0,l]$ such that $f\circ |\omega | = |\nu |\circ \phi$. If $k= 0$, $f^{\mathbb I}(\omega)$ is defined to be the path in $Q$ of length $0$ given by $f^{\mathbb I}(\omega)(0) = f_0(\omega(0))$. Note that if $g\colon P \to Q$ is a morphism of precubical sets such that $f = |g|$, then $f^{\mathbb I}(\omega) = g\circ \omega$. Note also that the path 
$f^{\mathbb I}(\omega )$ leads from $f_0(\omega(0))$ to $f_0(\omega(k))$ and that the map $f^{\mathbb I}\colon P^{\mathbb I} \to Q^{\mathbb I}$ is  compatible with composition and concatenation \cite[2.5]{weakmor}.

\subsection{Weak morphisms of HDAs }
 \label{defHDAmor}
A \emph{weak morphism} from an $M$-HDA  $\A = (P,I, F,\lambda)$ to an $M$-HDA $\B = (Q,J, G,\mu)$  is a weak morphism $f\colon |P| \to |Q|$  such that $f_0(I) \subseteq J$, $f_0(F) \subseteq G$ and  $\overline{\mu}\circ f^{\mathbb I} = \overline{\lambda}$. 

\subsection{Weak morphisms and precubical subsets} 
Weak morphisms are more flexible than morphisms of precubical sets, but they are much more rigid than arbitrary continuous maps. Here, we show that, like morphisms of precubical sets, they send precubical subsets to precubical subsets:

\begin{prop}
Let $f\colon |P| \to |Q|$ be a weak morphism of precubical sets and $X$ be a precubical subset  of $P$. Then there exists a unique precubical subset $A$ of $Q$ such that $f(|X|) = |A|$. It satisfies  $\dim(A) = \dim (X)$. If $X$ is finite, so is $A$.
\end{prop} 

\begin{proof}
Set $A = f_0(X_0) \cup \bigcup \limits _{x\in X_{>0}} x_{\flat}(R_x)$. Then $A$ is a precubical subset of $Q$ such $\dim(A) = \dim (X)$. If $X$ is finite, so is $A$. We have $X = X_0 \cup \bigcup \limits _{x\in X_{>0}} x_{\sharp}(\llbracket 0,1\rrbracket^{\otimes \deg(x)})$ and $$|X| = |X_0| \cup \bigcup \limits _{x\in X_{>0}} |x_{\sharp}(\llbracket 0,1\rrbracket^{\otimes \deg(x)})|.$$ Hence $f(|X|) = f(|X_0|) \cup \bigcup \limits _{x\in X_{>0}} f(|x_{\sharp}(\llbracket 0,1\rrbracket^{\otimes \deg(x)})|) = |f_0(X_0)| \cup \bigcup \limits _{x\in X_{>0}} |x_{\flat}(R_x)| = |A|$.

Suppose that $A'$ is another precubical subset of $Q$ such that $f(|X|) = |A'|$. Then $|A'| = |A|$. Consider an element $a \in A$. Then $[a,(\frac{1}{2}, \dots , \frac{1}{2})] \in |A'|$. It follows that there exist elements $a' \in A'$ and $u \in ]0,1[^{\deg(a')}$ such that $[a,(\frac{1}{2}, \dots , \frac{1}{2})] = [a',u]$. This implies that $a = a'$. Thus, $A \subseteq A'$. Similarly, $A' \subseteq A$.
\end{proof}

\begin{defin} Let $f\colon |P| \to |Q|$ be a weak morphism of precubical sets and $X$ be a precubical subset  of $P$. The unique precubical subset $A$ of $Q$ such that $f(|X|) = |A|$ is called the \emph{image of $X$ under $f$} and will be denoted by $f(X)$.
\end{defin}

\begin{rems}
(i) Let $g\colon P \to Q$ be a morphism of precubical sets and $X$ be a precubical subset of $P$. Then $|g|(X) = g(X)$.

(ii) Let $f\colon |P| \to |Q|$ be a weak morphism of precubical sets and $\{X_i\}_{i\in I}$ be a family of precubical subsets of $P$. Then $f(\bigcup_{i \in I}X_i) = \bigcup _{i\in I}f(X_i)$ and $f(\bigcap_{i \in I}X_i) \subseteq \bigcap _{i\in I}f(X_i)$. Given two precubical subsets $X, Y \subseteq P$ such that $X \subseteq Y$, one has $f(X) \subseteq f(Y)$.

(iii) Consider a weak morphism of precubical sets $f\colon |P| \to |Q|$ and an element $x \in P_n$. If $n >0$, then $f(x_{\sharp}(\llbracket 0,1\rrbracket ^{\otimes n})) = x_{\flat}(R_x)$. Indeed, $f(|x_{\sharp}(\llbracket 0,1\rrbracket ^{\otimes n})|) = |x_{\flat}(R_x)|$. If $n=0$, then $f(x_{\sharp}(\llbracket 0,1\rrbracket ^{\otimes n})) = \{f_0(x)\}$. Indeed, $f(|\{x\}|) = |\{f_0(x)\}|$.
\end{rems}

\section{The homology graph of an HDA} \label{hograph}

In this section, we define the homology graph of an HDA and establish its basic properties. We consider singular homology with coefficients in an arbitrary commutative unital ring $R$.

\subsection{The homology graph}
Let $P$ be a precubical set. We say that a homology class $\alpha \in H(|P|)$ \emph{points} to a homology class (of a possibly different degree) $\beta \in H(|P|)$  and write $\alpha \nearrow \beta$ if there exist precubical subsets $X, Y \subseteq P$ such that  $\alpha \in \im H(|X| \hookrightarrow |P|)$, $\beta \in \im H(|Y| \hookrightarrow |P|)$ and for all $x \in X_0$ and $y \in Y_0$ there exists a path in $P$ from $x$ to $y$. The \emph{homology graph} of $P$ is the directed graph whose vertices are the homology classes of $|P|$ and whose edges are given by the relation $\nearrow$. The \emph{homology graph} of an $M$-HDA $\A = (P,I, F,\lambda)$ is defined to be the homology graph of $P$.

\begin{exs}
The directed circle is the precubical set with exactly one vertex and exactly one edge. Every homology class of the directed circle points to every homology class.

The precubical set with two vertices $v$ and $w$ and two edges from $v$ to $w$ has the same homology as the directed circle but not the same homology graph. No non-trivial $1$-dimensional class points to a non-trivial $1$-dimensional class.

Another such example is depicted in figure \ref{fig3}. The precubical sets $P$ and $Q$ have the same homology but distinct homology graphs. In the case of $P$, the homology class representing the lower hole points to the homology class representing the upper hole. In the homology graph of $Q$, there are no edges between non-trivial homology classes of degree $1$.
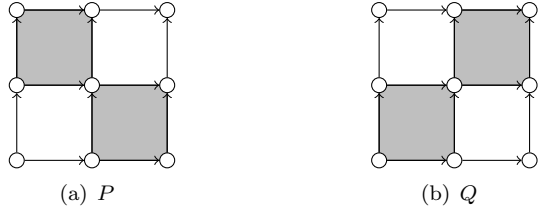
\begin{figure}
\subfigure[$P$]{
\begin{tikzpicture}[initial text={},on grid]

\path[draw, fill=lightgray] (0,1)--(1,1)--(1,2)--(0,2)--cycle
(1,0)--(2,0)--(2,1)--(1,1)--cycle;

 \node[state,minimum size=0pt,inner sep =2pt,fill=white] (q_0)   {};

\node[state,minimum size=0pt,inner sep =2pt,fill=white] (q_2) [right=of q_0] {};

      \node[state,minimum size=0pt,inner sep =2pt,fill=white] (q_6) [above=of q_0] {};

      \node[state,minimum size=0pt,inner sep =2pt,fill=white] (q_8) [above=of q_2] {};

      \node[state,minimum size=0pt,inner sep =2pt,fill=white] (q_11) [above=of q_6] {};

      \node[state,minimum size=0pt,inner sep =2pt,fill=white] (q_13) [above=of q_8] {};

   \node[state,minimum size=0pt,inner sep =2pt,fill=white] (p_1) [right=of q_2] {};

      \node[state,minimum size=0pt,inner sep =2pt,fill=white] (p_7) [above=of p_1] {};

      \node[state,minimum size=0pt,inner sep =2pt,fill=white] (p_12) [right=of q_13] {};

    \path[->] 
    (q_0) edge[above]  node {} (q_2)
    (q_2) edge[above]  node {} (p_1)
    
	(q_0) edge[left]  node {} (q_6)
    (q_2) edge[right]  node {} (q_8)
    (p_1) edge[left]  node {} (p_7)
    
    (q_6) edge[above]  node {} (q_8)
    (q_8) edge[above]  node {} (p_7)
    
    (q_6) edge[left]  node {} (q_11)
	(q_8) edge[right]  node {} (q_13)
    (p_7) edge[left]  node {} (p_12)
   
    (q_11) edge[above]  node {} (q_13)
    (q_13) edge[above]  node {} (p_12)
    ;
\end{tikzpicture}
}
\hspace{2cm}
\subfigure[$Q$]{
\begin{tikzpicture}[initial text={},on grid]

\path[draw, fill=lightgray] (1,1)--(2,1)--(2,2)--(1,2)--cycle
(0,0)--(1,0)--(1,1)--(0,1)--cycle;

 \node[state,minimum size=0pt,inner sep =2pt,fill=white] (q_0)   {};

\node[state,minimum size=0pt,inner sep =2pt,fill=white] (q_2) [right=of q_0] {};

      \node[state,minimum size=0pt,inner sep =2pt,fill=white] (q_6) [above=of q_0] {};

      \node[state,minimum size=0pt,inner sep =2pt,fill=white] (q_8) [above=of q_2] {};

      \node[state,minimum size=0pt,inner sep =2pt,fill=white] (q_11) [above=of q_6] {};

      \node[state,minimum size=0pt,inner sep =2pt,fill=white] (q_13) [above=of q_8] {};

   \node[state,minimum size=0pt,inner sep =2pt,fill=white] (p_1) [right=of q_2] {};

      \node[state,minimum size=0pt,inner sep =2pt,fill=white] (p_7) [above=of p_1] {};

      \node[state,minimum size=0pt,inner sep =2pt,fill=white] (p_12) [right=of q_13] {};

    \path[->] 
    (q_0) edge[above]  node {} (q_2)
    (q_2) edge[above]  node {} (p_1)
    
	(q_0) edge[left]  node {} (q_6)
    (q_2) edge[right]  node {} (q_8)
    (p_1) edge[left]  node {} (p_7)
    
    (q_6) edge[above]  node {} (q_8)
    (q_8) edge[above]  node {} (p_7)
    
    (q_6) edge[left]  node {} (q_11)
	(q_8) edge[right]  node {} (q_13)
    (p_7) edge[left]  node {} (p_12)
   
    (q_11) edge[above]  node {} (q_13)
    (q_13) edge[above]  node {} (p_12)
    ;
\end{tikzpicture}
}
\caption{Precubical sets with the same homology but different homology graphs}
\label{fig3}
\end{figure}

\end{exs}

\subsection{Some basic properties} 

Let $P$ be a precubical set.

\begin{prop}
Let $\alpha, \beta \in H(|P|)$ be homology classes such that $\alpha \nearrow \beta$. Then for all $r,s \in R$, $r\alpha \nearrow s \beta$.
\end{prop}

\begin{proof}
This is obvious.
\end{proof}

\begin{prop} \label{point0}
Consider a homology class $\alpha \in H(|P|)$ and the class $0$ of an arbitrary degree. Then  $0 \nearrow \alpha$ and $\alpha \nearrow 0$.
\end{prop}

\begin{proof}
This follows from the fact that $0 \in \im H(|\emptyset| \hookrightarrow |P|)$.
\end{proof}

As a consequence we obtain that the pointing relation is in general very far from being a partial order:

\begin{cor}
The relation $\nearrow$ is 
\begin{itemize}
\item[(i)] anti-symmetric if and only if $P = \emptyset$;
\item[(ii)] transitive if and only if $\alpha \nearrow \beta$ for all homology classes $\alpha, \beta \in H(|P|)$. 
\end{itemize} 
\end{cor}

\begin{exs}
(i) An example of a homology class that is not related to any non-trivial class is given in figure \ref{fig4}.

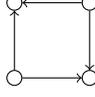
\begin{figure}
\begin{tikzpicture}[initial text={},on grid]

 \node[state,minimum size=0pt,inner sep =2pt,fill=white] (q_0)   {}; 
    
   \node[state,minimum size=0pt,inner sep =2pt,fill=white] (q_2) [right=of q_0,xshift=0cm] {};
   
   \node[state,minimum size=0pt,inner sep =2pt,fill=white] [above=of q_0, yshift=0cm] (q_3)   {};

   \node[state,minimum size=0pt,inner sep =2pt,fill=white] (q_5) [right=of q_3,xshift=0cm] {}; 
   
    \path[->] 
    (q_0) edge[below] node {} (q_2)
    (q_5) edge[above]  node {} (q_3)
    (q_0) edge[left]  node {} (q_3)
    (q_5) edge[right]  node {} (q_2);

\end{tikzpicture}
\caption{The homology class representing the hole is not incident with any edge from or to a non-trivial class} \label{fig4}
\end{figure}

(ii) In a precubical set with exactly one vertex, every homology class points to every homology class.
\end{exs}

\subsection{Compatibility with weak morphisms} Consider a weak morphism of precubical sets $f\colon |P| \to |Q|$.

\begin{theor} \label{morpoint}
Let $\alpha, \beta \in H(|P|)$ be homology classes such that $\alpha \nearrow \beta$. Then $f_*(\alpha) \nearrow f_*(\beta)$.
\end{theor}

\dem
Let $X, Y \subseteq P$ be precubical subsets such that $\alpha \in \im H(|X| \hookrightarrow |P|)$, $\beta \in \im H(|Y| \hookrightarrow |P|)$ and for all vertices $x \in X_0$ and $y \in Y_0$ there exists a path from $x$ to $y$. Then $f_*(\alpha) \in \im H(|f(X)| \hookrightarrow |Q|)$ and $f_*(\beta) \in \im H(|f(Y)| \hookrightarrow |Q|)$. 

Consider vertices $a \in f(X)_0$ and $b \in f(Y)_0$. We have to show that there exists a path in $Q$ from $a$ to $b$. Since $|f(X)|= f(|X|)$, there exist an integer $n$ and elements $x \in X_n$ and $u \in ]0,1[^n$ such that $[a,()] = f([x,u])$. We choose a path $\sigma$ in $Q$ beginning in $a$ as follows: If $n = 0$, then $\sigma$ is the constant path from $a$ to $a$. Suppose $n >0$. Let $R_x = \llbracket 0,k_1\rrbracket \otimes \cdots \otimes \llbracket 0,k_n\rrbracket$. Consider elements $\tilde a \in R_x$ and $w \in ]0,1[^ {\deg(\tilde a)}$ such that $\phi_x (u) = [\tilde a,w]$. Then $[a,()] = f\circ |x_{\sharp}|(u) = |x_{\flat}| \circ \phi_x (u) = [x_{\flat}(\tilde a),w]$. It follows that $\deg(\tilde a) = 0$ and $a = x_{\flat}(\tilde a)$. Let $\rho$ be a path in $R_x$ from $\tilde a$ to $(k_1, \dots, k_n)$ and set $\sigma = x_{\flat}\circ \rho$. Then $\sigma $ is a path in $Q$ from $a$ to $x_{\flat}(k_1, \dots, k_n)$. 

Since $|f(Y)| = f(|Y|)$, there exist an integer $m$ and elements $y \in Y_m$ and $u' \in ]0,1[^m$ such that $[b,()] = f([y,u'])$. In a similar fashion as above, we choose a path $\tau$ in $Q$ ending in $b$: If $m = 0$, then $\tau$ is the constant path from $b$ to $b$. Suppose $m >0$. Then there exists a vertex  $\tilde b \in R_y$ such that $y_{\flat}(\tilde b) = b$. Let $\theta$ be a path in $R_y$ from $(0, \dots, 0)$ to $\tilde b$ and set $\tau = y_{\flat} \circ \theta$. Then $\tau $ is a path in $Q$ from $y_{\flat}(0, \dots, 0)$ to $b$.

Consider the vertices $v \in X$ and $w\in Y$ defined by
$$v = \left\{\begin{array}{ll} x, & n = 0,\\x_{\sharp}(1,\dots,1), & n >0 \end{array}\right. \quad \mbox{and}\quad w = \left\{\begin{array}{ll} y, & m=0,\\ y_{\sharp}(0,\dots,0), & m>0.\end{array}\right.$$
Let $\omega $ be a path from $v$ to $w$. If $n=0$, then $f^{\mathbb I}(\omega)(0) = f_0(\omega (0)) = f_0(v) = f_0(x) = a$. If $n>0$, then $f^{\mathbb I}(\omega)(0) = f_0(v) = f_0(x_{\sharp}(1,\dots,1)) = x_{\flat}(k_1,\dots, k_n)$. The last equality holds because \begin{eqnarray*}[f_0(x_{\sharp}(1,\dots,1)),()] &=& f([x_{\sharp}(1,\dots,1),()])\\ &=& f\circ |x_{\sharp}|(1,\dots, 1)\\ &=& |x_{\flat}|\circ \phi_x (1,\dots, 1)\\ &=&  |x_{\flat}| (k_1,\dots, k_n)\\&=& [x_{\flat}(k_1,\dots, k_n),()].\end{eqnarray*} 
A similar argument shows that the end point of $f^{\mathbb I}(\omega)$ is $b$ if $m=0$ and $y_{\flat} (0,\dots ,0)$ else. It follows that $\sigma \cdot f^{\mathbb I}(\omega) \cdot \tau$ is a path in $Q$ from $a$ to $b$.
\findem

\section{Homeomorphic abstraction} \label{homeoabs}

We say that an $M$-HDA $\A = (P,I,F,\lambda)$ is a \emph{homeomorphic abstraction} of an $M\mbox{-}$HDA $\B = (Q,J,G,\mu)$, or that $\B$ is a \emph{homeomorphic refinement} of $\A$, if there exists a weak morphism $f$ from $\A$ to $\B$ that is a homeomorphism and satisfies $f_0(I) = J$ and $f_0(F) = G$.  For instance, in figure \ref{fig2}, $\A$ is a homeomorphic abstraction of $\B$. The purpose of this section is to show that the homology graph is invariant under homeomorphic abstraction. 

Throughout this section, $f\colon |P| \to |Q|$ is a weak morphism of precubical sets that is a homeomorphism.

\subsection{Carriers} 

The \emph{carrier} of an element $a \in Q$ with respect to $f$ is the unique element $c_f(a)\in P$ for which there exists an element $u \in ]0,1[^{\deg(c_f(a))}$ such that $f([c_f(a),u]) = [a,(\tfrac{1}{2},\dots,\tfrac{1}{2})]$. The \emph{carrier} of a precubical subset $A \subseteq Q$ with respect to $f$ is the precubical subset $c_f(A)$ of $P$ defined by $$c_f(A) = \bigcup \limits_{a\in A} c_f(a)_{\sharp}(\llbracket 0,1\rrbracket^{\otimes \deg(c_f(a))}).$$ We shall normally suppress the subscript $f$ and simply write $c(a)$ and $c(A)$ to denote carriers of elements and precubical subsets.

\begin{rems}
(i) If $f = |g|$ for a morphism of precubical sets $g\colon P \to Q$, then $g$ is an isomorphism and $g^{-1}$ is given by $g^{-1}(a) = c(a)$.

(ii) Let $\{A_i\}_{i\in I}$ be a family of precubical subsets of $Q$. Then $c(\bigcup_{i \in I}A_i) = \bigcup _{i\in I}c(A_i)$ and $c(\bigcap_{i \in I}A_i) \subseteq \bigcap _{i\in I}c(A_i)$. Given two precubical subsets $A, B \subseteq Q$ such that $A \subseteq B$, one has $c(A) \subseteq c(B)$.

(iii) For any element $a \in Q$,  $\deg(c(a)) \geq \deg(a)$. If $\deg(c(a)) = 0$, then $[a,(\frac{1}{2}, \dots, \frac{1}{2})] = f([c(a),()]) = [f_0(c(a)),()]$ and hence $\deg(a) = 0$. Suppose that $\deg(c(a))= n >0$. Let $u \in ]0,1[^{n}$ be the unique element such that $|c(a)_{\flat}|(\phi_{c(a)}(u)) = f([c(a),u]) = [a,(\tfrac{1}{2},\dots,\tfrac{1}{2})]$. Consider elements $z \in R_{c(a)}$ and $v \in ]0,1[^{\deg(z)}$ such that $\phi_{c(a)}(u) = [z,v]$. Then $[{c(a)}_{\flat}(z),v] = [a,(\frac{1}{2},\dots, \frac{1}{2})]$. Hence $a = {c(a)}_{\flat}(z)$ and therefore $\deg(a) \leq n$.
\end{rems}

The easy proof of the following proposition can be found in \cite{weakmor}:

\begin{prop} \label{cnat}
Consider a second  weak morphism of precubical sets  $g\colon |P'| \to |Q'|$ that is a homeomorphism and two morphisms of precubical sets $\xi \colon P' \to P$ and $\chi \colon Q'\to Q$ such that  $f\circ |\xi| = |\chi| \circ g$. Then for all $a \in Q'$, $c_f(\chi(a)) = \xi (c_g(a))$.
\end{prop}

%\begin{proof}
%Consider $u \in ]0,1[^{\deg(c(a))}$ such that $g([c(a),u]) = [a,(\frac{1}{2}, \dots, \frac{1}{2})]$. We have $f([\chi(c(a)),u]) = f\circ |\chi|([c(a),u]) = |\chi'|\circ g([c(a),u]) = |\chi'|([a,(\frac{1}{2}, \dots, \frac{1}{2})]) = [\chi'(a),(\frac{1}{2}, \dots, \frac{1}{2})]$ and hence $c(\chi'(a)) = \chi(c(a))$. 
%\end{proof}

Images and carriers are related as follows:

\begin{prop} \label{imcarr}
Consider precubical subsets of $X \subseteq P$ and $A \subseteq Q$ and an element $a \in Q$. Then
\begin{itemize}
\item[(i)] $a \in f(X)$ if and only if $c(a) \in X$;
\item[(ii)] $A \subseteq f(c(A))$;
\item[(iii)] $c(f(X)) = X$.
\end{itemize}
\end{prop}

\begin{proof} (i) Let $u \in ]0,1[$ be the unique element such that $f([c(a),u]) = [a,(\frac{1}{2}, \dots, \frac{1}{2})]$. Then we have $a \in f(X) \Leftrightarrow [a,(\frac{1}{2}, \dots, \frac{1}{2})] \in |f(X)| \Leftrightarrow f([c(a),u]) \in f(|X|) \Leftrightarrow [c(a),u] \in  |X| \Leftrightarrow c(a) \in X$.

(ii) Let $b \in A$. Since $c(b) \in c(A)$, by (i), $b \in f(c(A))$.

(iii) For $b \in f(X)$, $c(b) \in X$ and hence $c(b)_{\sharp} (\llbracket 0, 1 \rrbracket ^ {\otimes \deg (c(b))}) \subseteq X$. It follows that $c(f(X)) = \bigcup \limits_{b \in f(X)} c(b)_{\sharp} (\llbracket 0, 1 \rrbracket ^ {\otimes \deg (c(b))}) \subseteq X$. For the reverse inclusion, consider an element $x \in X$. Suppose first that $x$ is a vertex. Then $x = c(f_0(x)) \in c(f(X))$. Suppose now that $\deg (x) = n >0$.  Consider an element $z \in R_x$ such that $\deg(z) = n$. Then $c(z) = \iota_n$. Hence $x = x_{\sharp}(\iota_n) = x_{\sharp}(c(z)) = c(x_{\flat}(z))$. Since $x_{\flat}(R_x) = f(x_{\sharp}(\llbracket 0,1\rrbracket^{\otimes n}) \subseteq f(X)$, we have $x_{\flat}(z) \in f(X)$ and therefore  $x = c(x_{\flat}(z)) \in c(f(X))$. 
\end{proof}

\begin{prop}\cite{weakmor} \label{rho}
The map $f^{\mathbb I}\colon P^ {\mathbb I} \to Q^ {\mathbb I}$ admits a right inverse $\rho \colon Q^ {\mathbb I} \to P^ {\mathbb I}$. If $\omega$ is a path in $Q$ from $a$ to $b$, then $\rho(\omega )$ is a path in $P$ from $c(a)_{\sharp}(0,\dots, 0)$ to $c(b)_{\sharp}(0,\dots, 0)$.
\end{prop}

\begin{lem} \label{prelem}
Consider a  precubical subset $B$  of $Q$ and an element $b \in B$. Then $c(b_{\sharp}(0,\dots,0)) \in c(b)_{\sharp}(\llbracket 0,1 \llbracket^ {\otimes \deg(c(b))})$.
\end{lem}

\begin{proof}
Suppose first that $\deg(b) = 0$. Then $b_{\sharp}(0,\dots ,0) = b$ and therefore $$c(b_{\sharp}(0,\dots,0)) = c(b) = c(b)_{\sharp}(\iota_{\deg(c(b))}) \in c(b)_{\sharp}(\llbracket 0,1 \llbracket^ {\otimes \deg(c(b))}).$$

Suppose that $\deg(b) > 0$ and set $n = \deg(c(b))$. Then $b_{\sharp}(0,\dots,0) = d_1^0 \cdots d_1^0b$ and $n > 0$. Let $u \in ]0,1[^n$ be the unique element such that $$f([c(b),u]) = [b,(\tfrac{1}{2},\dots,\tfrac{1}{2})].$$ Consider elements $z \in R_{c(b)}$ and $v \in ]0,1[^{\deg (z)}$ such that $\phi_{c(b)}(u) = [z,v]$. Then $[c(b)_{\flat}(z),v] = |c(b)_{\flat}|\circ \phi_{c(b)}(u) = f\circ |c(b)_{\sharp}|(u) = f([c(b),u]) = [b,(\frac{1}{2},\dots,\frac{1}{2})]$ and hence $\deg(z) = \deg(b)$, $c(b)_{\flat}(z) = b$ and $v = (\frac{1}{2}, \dots, \frac{1}{2})$. Write $$z=(i_1,..,i_{s_1-1},[i_{s_1},i_{s_1}+1], i_{s_1+1},..,i_{s_2-1},[i_{s_2},i_{s_2}+1],.., [i_{s_r},i_{s_r}+1], i_{s_r+1},..,i_n).$$
Then \begin{eqnarray*}\phi_{c(b)}(u) &=& [z,(\tfrac{1}{2}, \dots, \tfrac{1}{2})]\\ 
&= & (i_1,..,i_{s_1-1},i_{s_1}+\tfrac{1}{2}, i_{s_1+1},..,i_{s_2-1},i_{s_2}+\tfrac{1}{2},.., i_{s_r}+\tfrac{1}{2}, i_{s_r+1},..,i_n).
\end{eqnarray*}
Write  $m = \deg(c(d_1^0 \cdots d_1^0z))$ and let $w \in ]0,1[^m$ be the unique element such that $\phi_{c(b)}([c(d_1^0 \cdots d_1^0z),w]) = [d_1^0 \cdots d_1^0z,()] = [z,(0,\dots, 0)] = (i_1, \dots, i_n)$. Since $\phi_{c(b)}$ is a dihomeomorphism, we have $[c(d_1^0 \cdots d_1^0z),w] \leq u$ and hence $[c(d_1^0 \cdots d_1^0z),w] \in  [0,1[^n$. By lemma  \ref{bbu} below, this implies that $c(d_1^0 \cdots d_1^0z) \in \llbracket 0,1 \llbracket^ {\otimes n}$. By  \ref{cnat}, $c(d_1^0 \cdots d_1^0b) = c(c(b)_{\flat}(d_1^0 \cdots d_1^0z)) = c(b)_{\sharp}(c(d_1^0 \cdots d_1^0z))$. The result follows. 
\end{proof}

\begin{lem} \label{bbu}
Consider elements $b \in \llbracket 0,{k_1} \rrbracket\otimes \cdots \otimes \llbracket 0,{k_n} \rrbracket$ $(n, k_1, \dots, k_n \geq 1)$ and $u \in ]0,1[^ {\deg(b)}$. Then 
\begin{itemize}
\item[(i)]
$ b \in  \rrbracket 0,{k_1} \llbracket \otimes \cdots \otimes \rrbracket 0,{k_n} \llbracket \;\;\; \Leftrightarrow [b,u] \in  ]0,k_1[ \times \cdots \times ]0,k_n[$; 
\item[(ii)] $b \in  \rrbracket 0,{k_1} \rrbracket \otimes \cdots \otimes \rrbracket 0,{k_n} \rrbracket    \; \Leftrightarrow  [b,u] \in  ]0,k_1] \times \cdots \times ]0,k_n]$; 
\item[(iii)] $b \in  \llbracket 0,{k_1} \llbracket  \otimes \cdots \otimes \llbracket 0,{k_n} \llbracket  \; \Leftrightarrow [b,u] \in  [0,k_1[ \times \cdots \times [0,k_n[$.
%\item[(iv)] $b \in  \partial(\llbracket 0,{l_1} \rrbracket  \otimes \cdots \otimes \llbracket 0,{l_n} \rrbracket)   \Leftrightarrow [b,u] \in  \partial([0,l_1]\times \cdots \times [0,l_n])$. 
\end{itemize}
\end{lem}

\begin{proof}
The last statement is \cite[4.7.1]{weakmor}. The proofs of the other statements are analogous.
%Write $b = (b_1, \dots, b_n)$ and suppose $\deg(b) = p$. Then there exist indices $1 \leq i_1 < \dots < i_p \leq n$ such that $\deg(b_{i_q}) = 1$ for $q \in \{1, \dots ,p\}$ and $b_i \in \{0, \dots, l_i\}$ for $i \notin \{i_1, \dots ,i_p\}$. For each $q \in \{1, \dots ,p\}$ there exists an element $j_q \in \{0,\dots, l_q-1\}$ such that $b_{i_q} = [j_q,j_q+1]$.  We have $[b,u] = (t_1, \dots, t_n) \in [0,l_1] \times \cdots \times [0,l_n]$ where $t_i = b_i$ for $i \notin \{i_1, \dots ,i_p\}$ and $t_{i_q} = j_q+u_q$ for $q \in \{1, \dots ,p\}$. The result follows.
\end{proof}

\subsection{Top cubes} 

We say that an element $z\in P$ is a \emph{top cube} of a precubical subset $X \subseteq P$ if $z\in X$ and there does not exist any element $x\in X$ such that $\deg(z) < \deg (x)$ and $z \in x_{\sharp}(\llbracket 0,1 \rrbracket^{\otimes \deg(x)})$.

\begin{lem} \label{maxa}
Let $A$ be a precubical subset of $Q$ and $z\in P$ be a top cube of $c(A)$. Then there exists an element $a \in A$ such that $z = c(a)$.
\end{lem}

\begin{proof}
Since $z \in c(A) = \bigcup \limits _{a \in A}c(a)_{\sharp}(\llbracket 0,1\rrbracket^{\otimes \deg(c(a))})$, there exists an element $a \in A$  such that $z \in c(a)_{\sharp}(\llbracket 0,1\rrbracket^{\otimes \deg(c(a))})$. Since $z$ is a top cube of $c(A)$, $\deg (z) \geq \deg (c(a))$. Thus, $z = c(a)$. 
\end{proof}

\begin{defin}
Let $A$ be a precubical subset of $Q$ and $z$ be a top cube of $c(A)$. We set $A^-_z = \{a\in A: c(a) \not=z\}$.
\end{defin}

\begin{prop} \label{canotin}
Let $A$ be a precubical subset of $Q$ and  $z$ be a top cube of $c(A)$. Then $A^-_z$ is a precubical subset of $A$ and $z \notin c(A^-_z)$. 
\end{prop}

\begin{proof}
Consider an element $a \in A^-_z$ such that $\deg(a) > 0$. Suppose that $d_i^ka \notin A^ -_z$. Then $c(d_i^ka) = z$. Since $c(a) \in c(a)_{\sharp}(\llbracket 0,1\rrbracket ^{\otimes \deg(c(a))})$, $a \in f(c(a)_{\sharp}(\llbracket 0,1\rrbracket ^{\otimes \deg(c(a))})) = c(a)_{\flat}(R_{c(a)})$. Write $a = c(a)_{\flat}(y)$. Then $$z = c(d_i^ka) = c(d_i^kc(a)_{\flat}(y)) = c(c(a)_{\flat}(d_i^ky)) = c(a)_{\sharp}(c(d_i^ky))\in c(a)_{\sharp}(\llbracket 0,1\rrbracket ^{\otimes \deg(c(a))}).$$ Since $z$ is a top cube of $c(A)$ and $c(a) \in c(A)$, we have $\deg(z) \geq \deg (c(a))$ and hence $z = c(a)$. This contradicts the fact that $a \in A^-_z$. It follows that $A^-_z$ is a precubical subset of $A$.

Suppose that $z \in c(A^-_z) = \bigcup \limits_{a \in A^-_z} c(a)_{\sharp}(\llbracket 0,1\rrbracket ^{\otimes \deg(c(a))})$. Then there exists an  element $a \in A^-_z$ such that $z \in c( a)_{\sharp}(\llbracket 0,1\rrbracket ^{\otimes \deg(c(a))})$. Since $a \in A^-_z$, $c(a) \not= z$ and therefore $\deg (z)  < \deg(c(a))$. This is impossible because $z$ is a top cube of $c(A)$.
\end{proof}

\begin{lem} \label{bcb}
Let $\phi \colon |\llbracket 0,1 \rrbracket ^{\otimes n}|  \to |\llbracket 0,k_1\rrbracket \otimes \cdots \otimes \llbracket 0, k_n\rrbracket|$ $(n>0)$ be a weak morphism that is a homeomorphism. Consider an element $b \in \llbracket 0,{k_1} \rrbracket\otimes \cdots \otimes \llbracket 0,{k_n} \rrbracket  $. Then $b \in  \rrbracket 0, k_1 \llbracket \otimes \cdots \otimes \rrbracket 0,k_n \llbracket$ if and only if $c(b) = \iota_n$.
\end{lem}

\begin{proof}
Let $u \in ]0,1[^{\deg(c(b))}$ be the element such that  $\phi([c(b),u]) = [b,(\frac{1}{2}, \dots, \frac{1}{2})]$. Since $\phi$ is a homeomorphism, $[b,(\frac{1}{2}, \dots, \frac{1}{2})] \in ]0,k_1[ \times \cdots \times ]0,k_m[$ if and only if $[c(b),u] \in ]0,1[^n$. Lemma \ref{bbu} implies the result.
\end{proof}

\begin{lem} \label{chiinj}
For any element $x\in P$ of degree $\geq 1$, the restriction of $x_{\flat}$ to  $\dot{R}_x$ is injective.
\end{lem}

\begin{proof}
Let $\deg(x) = n$ and $R_x = \llbracket 0,k_1\rrbracket \otimes \cdots \otimes \llbracket 0, k_n\rrbracket$. Let $a,b \in \dot{R}_x$ be distinct elements of the same degree. Then $[a,(\tfrac{1}{2}, \dots, \tfrac{1}{2})]\not=  [b,(\tfrac{1}{2}, \dots, \tfrac{1}{2})] \in ]0,k_1[\times \cdots \times ]0,k_n[$ and therefore 
$\phi_x^{-1}([a,(\tfrac{1}{2}, \dots, \tfrac{1}{2})])\not= \phi_x^ {-1}([b,(\tfrac{1}{2}, \dots, \tfrac{1}{2})]) \in ]0,1[^n.$ Let $u \in ]0,1[^{\deg(c(a))}$ and $v \in ]0,1[^{\deg(c(b))}$ be the uniquely determined elements such that $\phi_x ([c(a),u]) = [a,(\tfrac{1}{2}, \dots, \tfrac{1}{2})]$ and $\phi_x ([c(b),v]) = [b,(\tfrac{1}{2}, \dots, \tfrac{1}{2})]$. 
By \ref{bcb}, $c(a) =  c(b) = \iota _n$. Thus, $u \not=v$ and hence $|x_{\sharp}|([c(a),u])\not= |x_{\sharp}|([c(b),v])$. Since $f$ is injective, it follows that $|x_{\flat}|([a,(\tfrac{1}{2}, \dots, \tfrac{1}{2})])\not= |x_{\flat}|([b,(\tfrac{1}{2}, \dots, \tfrac{1}{2})])$ and hence that $x_{\flat}(a) \not= x_{\flat}(b)$.
\end{proof}

\begin{prop} \label{pushout}
Let $A$ be a precubical subset of $Q$ and  $z$ be a top cube of $c(A)$ of degree $n>0$. Then the diagram 
$$\xymatrix{z_{\flat}^{-1}(A^-_z) \ar^{z_{\flat}}[r] \ar@{^{(}->}+<0ex,-2ex>;[d] & A^-_z \ar@{^{(}->}+<0ex,-2ex>;[d]\\ z_{\flat}^{-1}(A) \ar^{z_{\flat}}[r]& A
}$$ is a push out of precubical sets.
\end{prop}

\begin{proof}
Let $R_z = \llbracket 0,k_1\rrbracket \otimes \cdots \otimes \llbracket 0, k_n\rrbracket$. Since $A$ is the disjoint union of $A^-_z$ and the graded set $B = \{a \in A : c(a) = z\}$,  $z_{\flat}^{-1}(A)$ is the disjoint union of $z_{\flat}^{-1}(A^-_z)$ and $z_{\flat}^{-1}(B)$. We have $z_{\flat} ^{-1}(B) \subseteq \dot{R}_z$. Indeed, if $y \in z_{\flat}^{-1}(B)$, then $z_{\sharp}(c(y)) = c(z_{\flat}(y)) = z$. This implies that $c(y) = \iota_n$ and hence, by lemma \ref{bcb}, that $y \in \rrbracket 0,k_1\llbracket  \otimes \cdots \otimes \rrbracket 0,k_n\llbracket$. By lemma \ref{chiinj}, it follows that $z_{\flat} \colon z_{\flat}^{-1}(B) \to B$ is injective. This map is also surjective because for $b \in B$, $c(b) = z \in z_{\sharp}(\llbracket 0,1\rrbracket^{\otimes n})$ and hence $b \in f(z_{\sharp}(\llbracket 0,1\rrbracket^{\otimes n})) = z_{\flat}(R_z)$. It follows that the diagram of the statement is a push out of graded sets. This implies that it is a push out of precubical sets.
\end{proof}

\subsection{Broken cubes}

Let $A$ be a precubical subset of $Q$. An $n$-cube $x\in P$ is said to be \emph{broken in $A$} if $x \in c(A)$  and $f(x_{\sharp}(\llbracket 0,1\rrbracket^ {\otimes n})) \not \subseteq A$.

\begin{prop} \label{subbroken}
Let $z$ be a top cube of $c(A)$ and $x \in P$ be an $n$-cube that is broken in $A^-_z$. Then $x$ is broken in $A$. 
\end{prop}

\begin{proof}
Suppose that $x$ is not broken in $A$. Since $x \in c(A^-_z) \subseteq c(A)$, we have $f(x_{\sharp}(\llbracket 0,1\rrbracket^ {\otimes n})) \subseteq A$. It follows that there exists an element $b \in f(x_{\sharp}(\llbracket 0,1\rrbracket^ {\otimes n}))$ such that $b \not \in A^-_z$ but $b \in A$. Since $b \in f(x_{\sharp}(\llbracket 0,1\rrbracket^ {\otimes n}))$, $c(b) \in x_{\sharp}(\llbracket 0,1\rrbracket^ {\otimes n})$.  Since $b \notin A^-_z$, $c(b) = z$. It follows that $z \in x_{\sharp}(\llbracket 0,1\rrbracket^ {\otimes n}) \subseteq c(A^-_z)$. By proposition \ref{canotin}, this is impossible.
\end{proof}

\begin{prop} \label{brokenmax}
Suppose that $A$ is finite and that some element of $P$ is broken in $A$. Then there exists top cube of $c(A)$ that is broken in $A$. 
\end{prop}

\begin{proof}
Let $x \in P$ be broken in $A$. Since $A$ is finite, so is $$c(A) = \bigcup \limits _{a \in A}c(a)_{\sharp}(\llbracket 0,1\rrbracket^{\otimes \deg(c(a))}).$$ It follows that there exists a top cube $z$ of $c(A)$ such that $x \in z_{\sharp}(\llbracket 0,1 \rrbracket ^{\deg (z)})$. Thus,  $x_{\sharp}(\llbracket 0,1\rrbracket ^{\deg(x)}) \subseteq z_{\sharp}(\llbracket 0,1 \rrbracket ^{\deg (z)})$, and hence $f(x_{\sharp}(\llbracket 0,1\rrbracket ^{\deg(x)})) \subseteq f(z_{\sharp}(\llbracket 0,1 \rrbracket ^{\deg (z)}))$. Since $x$ is broken in $A$, we have $f(x_{\sharp}(\llbracket 0,1\rrbracket ^{\deg(x)})) \not \subseteq A$ and hence also $$f(z_{\sharp}(\llbracket 0,1 \rrbracket ^{\deg (z)})) \not \subseteq A.$$ Thus, $z$ is broken in $A$.
\end{proof}

\begin{prop} \label{degnot0}
Let $z$ be a top cube of $c(A)$ that is broken in $A$. Then $\deg(z) > 0$.
\end{prop}

\begin{proof}
Suppose that $\deg(z) = 0$. By lemma \ref{maxa}, there exists an element $a \in A$ such that $z = c(a)$. Since $\deg(c(a)) = 0$, also $\deg(a) = 0$. We have $[f_0(c(a)),()] = f([c(a),()]) = [a,()]$ and hence $f_0(z) = f_0(c(a)) = a$. Thus, $f(z_{\sharp}(\llbracket 0,1\rrbracket^{\deg(z)})) =  \{f_0(z)\} \subseteq A$. This contradicts the assumption that $z$ is broken.
\end{proof}

\begin{prop} \label{nobroken}
Suppose that no element of $c(A)$ is broken in $A$. Then  $A = f(c(A))$.
\end{prop}

\begin{proof}
By \ref{imcarr}, $A \subseteq f(c(A))$. Since no element of $c(A)$ is broken in $A$, $f(c(A)) = f(\bigcup \limits_{x \in c(A)} x_{\sharp}(\llbracket 0,1\rrbracket^{\otimes \deg(x)})) = \bigcup \limits_{x \in c(A)} f(x_{\sharp}(\llbracket 0,1\rrbracket^{\otimes \deg(x)})) \subseteq A$. 
\end{proof}

\subsection{Past-complete cubes} 
Let $A$ be a precubical subset of $Q$. Consider an $n\mbox{-}$cube $x\in c(A)$.  We say that $x$ is \emph{past-complete in $A$} if either $n = 0$ or for all vertices $(l_1, \dots, l_n) \in R_x$, $$x_{\flat} (l_1, \dots, l_n) \in A \Rightarrow x_{\flat}(\llbracket 0,l_1\rrbracket \otimes \cdots \otimes \llbracket 0, l_n\rrbracket) \subseteq A.$$

\begin{prop} \label{allpastcomplete}
Let $z$ be a top cube of $c(A)$. If all elements of $c(A)$ are past-complete in $A$, then all elements of $c(A^-_z)$ are past-complete in $A^-_z$.
\end{prop}

\begin{proof}
Suppose that there exists an $n$-cube  $x \in c(A^-_z)$ that is not past-complete in $A^-_z$. Consider a vertex $(l_1, \dots , l_n) \in R_x$ such that $x_{\flat}(l_1, \dots, l_n) \in A^-_z$ but $x_{\flat}(\llbracket 0,l_1\rrbracket \otimes \cdots \otimes \llbracket 0, l_n\rrbracket) \not \subseteq A^-_z$. Consider an element $y \in \llbracket 0,l_1\rrbracket \otimes \cdots \otimes \llbracket 0, l_n\rrbracket$ such that $x_{\flat}(y) \notin A^-_z$. Since $x$ is past-complete in $A$, $x_{\flat}(y) \in A$. Hence $c(x_{\flat}(y)) = z$. Since $x_{\flat}(y) \in f(x_{\sharp}(\llbracket 0,1\rrbracket ^{\otimes n}))$, $c(x_{\flat}(y)) \in x_{\sharp}(\llbracket 0,1\rrbracket ^{\otimes n})$. It follows that $z \in x_{\sharp}(\llbracket 0,1\rrbracket^ {\otimes n}) \subseteq c(A^-_z)$. By proposition \ref{canotin}, this is impossible.
\end{proof}

\subsection{Order-convex sets} \label{orderconvex}
A subset $B$ of a partially ordered set $(S,\leq)$ is called \emph{order-convex} if for all $x,y \in B$, $\{z \in S : y\leq z \leq x\} \subseteq B$. A subset $B$ of $S$ with a minimal element $m$ is order-convex if and only if for each element  $x \in B$, $\{z \in S : m \leq z \leq x\} \subseteq B$.

Let $B$ be an order-convex subset of $[0,k_1] \times \cdots \times [0,k_n]\setminus \{(k_1, \dots, k_n)\}$ such that $(0,\dots,0) \in B$. Then $B$ is contractible. A contraction  $G\colon B \times [0,1] \to B$ is given by  $G(x,t) = (1-t)x$.

Write $\partial ([0,k_1] \times \cdots \times [0,k_n]) = ([0,k_1] \times \cdots \times [0,k_n]) \setminus (]0,k_1[ \times \cdots \times ]0,k_n[)$. Let $B$ be an order-convex subset of $\partial ([0,k_1] \times \cdots \times [0,k_n]) \setminus \{(k_1, \dots, k_n)\}$ such that $(0, \dots, 0) \in B$. Then $B$ is contractible. A contraction $H\colon B \times [0,1] \to B$ is given by 
$$H(x,t) = \left\{\begin{array}{ll} x - \tfrac{2t\cdot\min \limits_{1\leq i \leq n}\frac{x_i}{k_i}}{1-\min \limits_{1\leq i \leq n}\tfrac{x_i}{k_i}} ((k_1, \dots, k_n)-x), & t \leq \frac{1}{2},\\& \\ (2-2t)(x - \tfrac{\min \limits_{1\leq i \leq n}\frac{x_i}{k_i}}{1-\min \limits_{1\leq i \leq n}\tfrac{x_i}{k_i}} ((k_1, \dots, k_n)-x)), & t\geq \frac{1}{2}.\end{array}\right.$$

Let $B$ be an order-convex subset of $[0,k_1] \times \cdots \times [0,k_n] \setminus\{(k_1, \dots, k_n)\}$ such that $(0,\dots 0) \in B$. Then $B\cap \partial ([0,k_1] \times \cdots \times [0,k_n])$ is an order-convex subset of $\partial ([0,k_1] \times \cdots \times [0,k_n]) \setminus \{(k_1, \dots, k_n)\}$ and $(0, \dots ,0) \in B\cap \partial ([0,k_1] \times \cdots \times [0,k_n])$. Hence $B\cap \partial ([0,k_1] \times \cdots \times [0,k_n])$ is contractible.

\subsection{A homotopy equivalence} Let $A$ be a precubical subset of $Q$ and $z$ be a top cube of $c(A)$ that is broken but past-complete in $A$. We shall show that the inclusion $|A^-_z| \hookrightarrow |A|$ is a homotopy equivalence. Set $n = \deg(z)$. By \ref{degnot0}, $n >0$. Let  $R_z = \llbracket 0, k_1\rrbracket \otimes \cdots \otimes \llbracket 0, k_n\rrbracket$. 

\begin{lem} \label{0in}
$(0, \dots, 0) \in z_{\flat}^{-1}(A)$.
\end{lem}

\begin{proof}
By lemma \ref{maxa}, there exists an element $a \in A$ such that $z = c(a)$. We have $a \in  f(c(a)_{\sharp}(\llbracket 0,1\rrbracket ^{\otimes n})) = z_{\flat}(R_z)$. It follows that there exists an element $y \in R_z$ such that $z_{\flat} (y) = a$. Since $a\in A$, $y \in z_{\flat}^{-1}(A)$ and therefore $z_{\flat}^{-1}(A) \not= \emptyset$. Let $(l_1, \dots , l_n)$ be a vertex in $z_{\flat}^{-1}(A)$. Since $z$ is past-complete in $A$, $z_{\flat} (\llbracket 0, l_1\rrbracket \otimes \cdots \otimes \llbracket 0, l_n\rrbracket) \subseteq A$. This implies that $(0, \dots, 0) \in z_{\flat}^{-1}(A)$.
\end{proof}

\begin{lem} \label{knotin}
$(k_1, \dots, k_n) \notin z_{\flat}^{-1}(A)$.
\end{lem}

\begin{proof}
Suppose that $(k_1, \dots, k_n) \in z_{\flat}^{-1}(A)$. Since $z$ is past-complete in $A$, $$f(z_{\sharp}(\llbracket 0,1\rrbracket^{\otimes n}))= z_{\flat} (\llbracket 0, k_1\rrbracket \otimes \cdots \otimes \llbracket 0, k_n\rrbracket) \subseteq A.$$ This contradicts the fact that $z$ is broken in $A$. 
\end{proof}

\begin{lem} \label{chiorder}
$|z_{\flat}^{-1}(A)|$ is an order-convex subset of $$[0,k_1] \times \cdots \times [0,k_n] \setminus\{(k_1, \dots, k_n)\}$$ and $(0, \dots , 0) \in |z_{\flat}^{-1}(A)|$.
\end{lem}

\begin{proof}
By \ref{0in} and \ref{knotin}, $|z_{\flat}^{-1}(A)|$ is a subset of $[0,k_1] \times \cdots \times  [0,k_n] \setminus\{(k_1, \dots, k_n)\}$ that contains $(0, \dots, 0)$. It remains to show that $|z_{\flat}^{-1}(A)|$ is order-convex. Consider an element $[y, u] \in |z_{\flat}^{-1}(A)|$ where $y  \in z_{\flat}^{-1}(A)$ and $u \in ]0,1[^ {\deg (y)}$. Write $y = (y_1, \dots, y_n)$ and suppose $\deg(y) = p$. Then there exist indices $1 \leq i_1 < \dots < i_p \leq n$ such that $\deg(y_{i_q}) = 1$ for $q \in \{1, \dots ,p\}$ and $y_i \in \{0, \dots, k_i\}$ for $i \notin \{i_1, \dots ,i_p\}$. For each $q \in \{1, \dots ,p\}$ there exists an element $j_q \in \{0,\dots, k_q-1\}$ such that $y_{i_q} = [j_q,j_q+1]$.  
We have $$[y,u] = [y,(u_1,\dots,u_p)] = (t_1, \dots, t_n) \in [0,k_1] \times \cdots \times [0,k_n]$$ where $t_i = y_i$ for $i \notin \{i_1, \dots ,i_p\}$ and $t_{i_q} = j_q+u_q$ for $q \in \{1, \dots ,p\}$. Set $l_i = y_i$ for $i \notin \{i_1, \dots ,i_p\}$ and $l_{i_q} = j_q+1$ for $q \in \{1, \dots ,p\}$. Then $(l_1, \dots, l_n) = [y,(1,\dots,1)] \in y_{\sharp}(\llbracket 0,1\rrbracket^{\otimes p}) \subseteq z_{\flat}^{-1}(A)$ and $[y,u] \leq (l_1, \dots, l_n)$. Since $z$ is past-complete in $A$, $\llbracket 0, l_1\rrbracket \otimes \cdots \otimes \llbracket 0, l_n\rrbracket \subseteq z_{\flat}^{-1}(A)$ and hence $[0, l_1] \times \cdots \times  [0, l_n] \subseteq |z_{\flat}^{-1}(A)|$. Consider an element $x \in [0,k_1] \times \cdots \times [0,k_n]\setminus \{(k_1, \dots, k_n)\}$ such that $x \leq [y,u]$. Then $x \leq (l_1, \dots, l_n)$ and therefore $x \in [0, l_1] \times \cdots \times  [0, l_n] \subseteq |z_{\flat}^{-1}(A)|$. 
\end{proof}

\begin{lem} \label{boundary}
$|z_{\flat}^{-1}(A^-_z)| = |z_{\flat}^{-1}(A)|\cap \partial ([0,k_1] \times \cdots \times [0,k_n])$.
\end{lem}

\begin{proof}
We have $|\partial R_z| = \partial ([0,k_1] \times \cdots \times [0,k_n])$. It is therefore enough to show that $z_{\flat}^{-1}(A^-_z) = z_{\flat}^{-1}(A)\cap  \partial R_z$. Consider an element $y \in z_{\flat}^{-1}(A)$. We have $y\in z_{\flat}^{-1}(A^-_z) \Leftrightarrow z_{\flat}(y) \in A^-_z \Leftrightarrow c(z_{\flat} (y)) \not= z  \Leftrightarrow z_{\sharp}(c(y)) \not= z \Leftrightarrow c(y) \not= \iota_n \Leftrightarrow y \in \partial R_z$. 
\end{proof}

\begin{prop} \label{hoequiv}
The inclusion $|A^-_z| \hookrightarrow |A|$ is a homotopy equivalence. 
\end{prop}

\begin{proof}
As a left adjoint, the geometric realisation preserves push outs. Hence, by \ref{pushout}, the diagram
$$\xymatrix{|z_{\flat}^{-1}(A^-_z)| \ar^{|z_{\flat}|}[r] \ar@{^{(}->}+<0ex,-2ex>;[d] & |A^-_z| \ar@{^{(}->}+<0ex,-2ex>;[d]\\ |z_{\flat}^{-1}(A)| \ar^{|z_{\flat}|}[r]& |A|
}$$
is a push out of topological spaces. By \ref{orderconvex}, \ref{chiorder} and \ref{boundary}, $|z_{\flat}^{-1}(A)|$ and $|z_{\flat}^{-1}(A^-_z)|$ are contractible. Therefore the inclusion $|z_{\flat}^{-1}(A^-_z)| \hookrightarrow |z_{\flat}^{-1}(A)|$ is a homotopy equivalence. Since it is the inclusion of a sub CW-complex, it is a closed cofibration. The result follows.
\end{proof}

\subsection{Invariance of the homology graph under homeomorphic abstraction} It follows from theorem \ref{main} below that if an $M$-HDA $\A$ is a homeomorphic abstraction of an $M$-HDA $\B$, then $\A$ and $\B$ have isomorphic homology graphs.

\begin{lem} \label{reduction}
Let $A$ be a finite precubical subset of $Q$ such that all elements of $c(A)$ are past complete in $A$. Then there exists a precubical subset $\tilde A$ of $A$ such that the inclusion $|\tilde A| \hookrightarrow |A|$ is a homotopy equivalence and no element of $P$ is broken in $\tilde A$.
\end{lem}

\begin{proof}
Suppose that the result is not true. Let $B$ be a precubical subset of $A$ such that all elements of $c(B)$ are past-complete in $B$, the inclusion $|B| \hookrightarrow |A|$ is a homotopy equivalence and the number $q$ of elements of $P$ that are  broken in $B$ is minimal. By the hypothesis, $q > 0$. By proposition \ref{brokenmax}, there exists a top cube $z \in c(B)$ that is broken in $B$. By proposition \ref{allpastcomplete}, all elements of $c(B^-_z)$ are past-complete in $B^-_z$. By proposition \ref{hoequiv}, the inclusion $|B^-_z| \hookrightarrow |B|$ is a homotopy equivalence. It follows that the inclusion $|B^-_z| \hookrightarrow |A|$ is a homotopy equivalence. Since $A$ is finite, so are $B$, $B^-_z$, $c(B)$ and $c(B^-_z)$. Hence $q$ and the number $r$ of elements of $P$ that are broken in $B^-_z$ are finite. By proposition \ref{subbroken}, any element of $P$ that is broken in $B^-_z$ is also broken in $B$. By proposition \ref{canotin}, $z \notin c(B^-_z)$. Hence $z$ is broken in $B$ but not in $B^-_z$. It follows that $r < q$. This contradicts the minimality of $q$.
\end{proof}

\begin{lem} \label{lemma1}
Consider homology classes $\alpha, \beta \in H(|P|)$. Let $X \subseteq P$ and $B \subseteq Q$ be precubical subsets such that $\alpha \in \im H(|X| \hookrightarrow |P|)$, $f_*(\beta) \in \im H(|B| \hookrightarrow |Q|)$ and for all vertices $a \in f(X)_0$ and $b\in B_0$ there exists a path in $Q$ from $a$ to $b$. Then $\alpha \nearrow \beta$.
\end{lem}

\begin{proof}
Since $B \subseteq f(c(B))$, we have $f_*(\beta) \in \im H(|f(c(B))| \hookrightarrow |Q|)$. Therefore $\beta \in \im H(|c(B)| \hookrightarrow |P|).$ Consider vertices $x \in X_0$ and $y \in c(B)_0$. Consider $b \in B$ such that $y \in c(b)_{\sharp}(\llbracket 0,1\rrbracket^{\otimes \deg(c(b))})$. Let $\omega$ be a path in $Q$ from $f_0(x)$ to $b_{\sharp}(0,\dots,0)$. Then, by \ref{rho},  $\rho(\omega)$ is a path in $P$ from $c(f_0(x)) = x$ to $c(b_{\sharp}(0,\dots,0))_{\sharp}(0,\dots ,0)$. By \ref{prelem}, $c(b_{\sharp}(0,\dots,0)) \in c(b)_{\sharp}(\llbracket 0,1 \llbracket^ {\otimes \deg(c(b))})$. Since $\llbracket 0,1 \llbracket ^{\otimes \deg(c(b))}$ is closed under the operators $d_i^0$, so is $c(b)_{\sharp}(\llbracket 0,1 \llbracket^ {\otimes \deg(c(b))})$. It follows that $c(b_{\sharp}(0,\dots,0))_{\sharp}(0,\dots ,0) \in c(b)_{\sharp}(\llbracket 0,1 \llbracket^ {\otimes \deg(c(b))})$ and hence that  $$c(b_{\sharp}(0,\dots,0))_{\sharp}(0,\dots ,0) = c(b)_{\sharp}(0,\dots ,0).$$ Let $\nu$ be a path in $P$ from $c(b)_{\sharp}(0,\dots ,0)$ to $y$. Then $\rho(\omega)\cdot \nu$ is a path in $P$ from $x$ to $y$. It follows that $\alpha \nearrow \beta$.
\end{proof}

\begin{theor} \label{main}
For all homology classes $\alpha, \beta \in H(|P|)$, $\alpha \nearrow \beta$ if and only if $f_*(\alpha) \nearrow f_*(\beta)$.
\end{theor}

\begin{proof}
We only have to show the if part of the statement. Consider homology classes $\alpha, \beta \in H(|P|)$ such that $f_*(\alpha) \nearrow f_*(\beta)$. Let $A$ and $B$ be precubical subsets of $Q$ such that $f_*(\alpha) \in \im H(|A| \hookrightarrow |Q|)$, $f_*(\beta) \in \im H(|B| \hookrightarrow \nolinebreak |Q|)$ and for all vertices $a \in A_0$ and $b \in B_0$ there exists a path in $Q$ from $a$ to $b$. We may suppose that $A$ is finite. Then also $f(c(A))$ is finite. Let $A'$ be the largest precubical subset of $f(c(A))$ such that $A \subseteq A'$  and for all vertices $a \in A'_0$ and $b \in B_0$ there exists a path in $Q$ from $a$ to $b$. Then $A'$ is finite and $f_*(\alpha) \in \im H(|A'| \hookrightarrow |Q|)$. We have $c(A) \subseteq c(A') \subseteq c(f(c(A))) = c(A)$ and hence $c(A') = c(A)$. 

We show that all elements of $c(A')$ are past-complete in $A'$. Suppose that there exists an element $x \in c(A')$ that is not past-complete in $A'$.  Then $\deg(x) = n > 0$ and there exists a vertex $(l_1, \dots , l_n) \in R_x$ such that $x_{\flat}(l_1, \dots, l_n) \in A'$ but $x_{\flat}(\llbracket 0,l_1\rrbracket \otimes \cdots \otimes \llbracket 0, l_n\rrbracket) \not \subseteq A'$. Since $x_{\flat}(l_1, \dots, l_n) \in A'$, for all vertices $a \in x_{\flat}(\llbracket 0,l_1\rrbracket \otimes \cdots \otimes \llbracket 0, l_n\rrbracket)$ and $b \in B$, there exists a path in $Q$ from $a$ to $b$. Since $x_{\flat}(\llbracket 0,l_1\rrbracket \otimes \cdots \otimes \llbracket 0, l_n\rrbracket) \subseteq x_{\flat}(R_x) = f(x_{\sharp}(\llbracket 0,1 \rrbracket ^{\otimes n})) \subseteq f(c(A')) = f(c(A))$, the maximality of $A'$ implies that $A' \cup x_{\flat}(\llbracket 0,l_1\rrbracket \otimes \cdots \otimes \llbracket 0, l_n\rrbracket) \subseteq A'$ and hence that $x_{\flat}(\llbracket 0,l_1\rrbracket \otimes \cdots \otimes \llbracket 0, l_n\rrbracket) \subseteq A'$, a contradiction. It follows that all elements of $c(A')$ are past-complete in $A'$.

By lemma \ref{reduction}, there exists a precubical subset $\tilde A$ of $A'$ such that the inclusion $|\tilde A| \hookrightarrow |A'|$ is a homotopy equivalence and no element of $P$ is broken in $\tilde A$. Since $\tilde A \subseteq A'$, for all vertices $a \in \tilde A_0$ and $b \in B_0$ there exists a path in $Q$ from $a$ to $b$. Since the inclusion $|\tilde A| \hookrightarrow |A'|$ is a homotopy equivalence, $f_*(\alpha) \in \im H(|\tilde A| \hookrightarrow |Q|)$. Since no element of $P$ is broken in $\tilde A$, by propsition \ref{nobroken}, $f(c(\tilde A)) = \tilde A$. It follows that $\alpha \in \im H(|c(\tilde A)| \hookrightarrow |P|)$ and consequently, by lemma \ref{lemma1}, that $\alpha \nearrow \beta$.
\end{proof}


\begin{thebibliography}{} \label{biblio}
%\bibitem{BaierKatoen} C. Baier, J.-P. Katoen, Principles of Model Checking, The MIT Press (2008).

%\bibitem{Baues} H.J. Baues: \emph{Algebraic Homotopy}, Cambridge University Press (1989).


%\bibitem{Bubenik} P. Bubenik, Context for models of concurrency, in \emph{Preliminary Proceedings \linebreak of the Workshop on Geometry and Topology in Concurrency and Distributed Computing GETCO 2004}, vol NS-04-2 of \emph{BRICS Notes}, pp. 33-49. BRICS, Amsterdam, The Netherlands, 2004.

%\bibitem{BubenikW} P. Bubenik and K. Worytkiewicz, A model category for local po-spaces, \emph{Homology, Homotopy and Applications} 8 (2006),  263-292.

%\bibitem{BubenikExtremal} P. Bubenik, Models and Van Kampen theorems for directed homotopy theory, Homology, Homotopy and Applications 11 (1) (2009), 185-202.

%\bibitem{ClarkeGrumbergPeled} E.M. Clarke, Jr., O. Grumberg, D.A. Peled, \emph{Model Checking}, The MIT Press (1999).

%\bibitem{Coffman} E.G. Coffman, M.J. Elphick, and A. Shoshani, \emph{System Deadlocks}, Computing Surveys 3 (1971), 67-78.

%\bibitem{Diekert} V. Diekert, Combinatorics on Traces, Lecture Notes in Computer Science 454, Springer-Verlag (1990).

%\bibitem{Traces} V. Diekert, G. Rozenberg (eds.), \emph{The Book of Traces}, World Scientific (1995).

%\bibitem{DwyerSpalinsky} W.G. Dwyer and J. Spalinsky, Homotopy Theories and Model Categories, \emph{Handbook of Algebraic  Topology}, Elsevier (1995), 73-126. 

\bibitem{FahrenbergDiH} U. Fahrenberg, Directed Homology, Electronic Notes in Theoretical Computer Science 100 (2004), 111-125.

\bibitem{FahrenbergThesis} U. Fahrenberg, Higher dimensional automata from a topological viewpoint, Ph.D. thesis, Aal\nolinebreak borg University, Denmark (2005).

%\bibitem{FahrenbergRaussen} U. Fahrenberg, M. Rau\ss en, Reparametrizations of continuous paths, \emph{Journal of Homotopy and Related Structures} 2(2) (2007), 93-117.

%\bibitem{Fajstrup} L. Fajstrup, Dipaths and dihomotopies in a cubical complex, Advances in Applied Mathematics 35 (2005), 188-206.

\bibitem{FajstrupGR} L. Fajstrup, M. Rau\ss en, E.  Goubault, Algebraic topology and concurrency, Theoretical Computer Science 357 (2006), 241-278.

%\bibitem{Components} L. Fajstrup, M. Rau\ss en, E. Goubault, E. Haucourt, Components of the Fundamental Category, \emph{Applied Categorical Structures} 12 (2004), 81-108.

%\bibitem{Gaucher} P. Gaucher, A model category for the homotopy theory of concurrency,  \emph{Homology, Homotopy and Applications} 5(2) (2003), 549-599.

%\bibitem{GaucherNoModel} P. Gaucher, Flow does not model flows up to weak dihomotopy, \emph{Applied Categorical Structures} 13 (2005), pp. 371-388.

\bibitem{GaucherHomol} P. Gaucher, Homological properties of non-deterministic branchings and mergings in higher dimensional automata, Homology, Homotopy and Applications 7(1) (2005), 51-76.


%\bibitem{GaucherProcess} P. Gaucher, Towards a homotopy theory of process algebra, Homology, Homotopy and Applications 10 (1) (2008), 353-388.

\bibitem{GaucherGoubault} P. Gaucher, E. Goubault, Topological deformation of higher dimensional automata, Homology, Homotopy and Applications 5 (2) (2003), 39-82.

\bibitem{vanGlabbeek} R.J. van Glabbeek, On the expressiveness of higher dimensional automata, Theoretical Computer Science 356 (2006), 265-290.


%\bibitem{GoerssJardine}  P. Goerss and J.F. Jardine, \emph{Simplicial Homotopy Theory}, Progress in Mathematics 174, Birkh\"auser (1999).

%\bibitem{EG} E. Goubault, \emph{The Geometry of Concurrency}, Ph.D.thesis, \'Ecole Normale Sup\'erieure (1995).

\bibitem{Goubault} E. Goubault, Some geometric perspectives in concurrency theory, Homology, Homotopy and Applications 5 (2) (2003), 95-136.



%\bibitem{GoubaultTransCub} E. Goubault, Cubical Sets are Generalized Transition Systems, preprint (2002), available at http://citeseerx.ist.psu.edu/viewdoc/summary?doi=10.1.1.23.7908

%\bibitem{ALCOOL} E. Goubault, E. Haucourt, A Practical Application of Geometric Semantics to Static Analysis of Concurrent Programs, \emph{M. Abadi and L. de Alfaro (Eds.): CONCUR 2005}, Springer LNCS 3653 (2005), 503-517.

%\bibitem{Components2} E. Goubault, E. Haucourt, Components of the Fundamental Category II, \emph{Applied Categorical Structures} 15 (2007), 387-414.


\bibitem{GoubaultJensen} E. Goubault, T.P. Jensen, Homology of higher dimensional automata, in: R. Cleaveland (Ed.), Proc. CONCUR '92, Third Internat. Conf. on Concurrency Theory, Stony Brook, NY, USA, August 1992, Lecture Notes in Computer Science 630, Springer-Verlag (1992), 254-268.

%\bibitem{Grandis} M. Grandis, Directed homotopy theory. I.  \emph{Cah. Topol. G\'eom. Diff\'er. Cat\'eg.}  44 (4) (2003),  281--316.

\bibitem{GrandisBook} M. Grandis, Directed Algebraic Topology - Models of Non-Reversible Worlds, New Mathematical Monographs 13, Cambridge University Press (2009).

%\bibitem{Holzmann} G.J. Holzmann, \emph{The SPIN model checker: primer and reference manual}, Addison-Wesley (2008).

\bibitem{Husainov} A.A. Husainov,
Homology and Bisimulation of Asynchronous Transition Systems and Petri Nets, arXiv:1307.5377v1 (2013).


%\bibitem{reldi} Thomas Kahl, Relative directed homotopy theory of partially ordered spaces, Journal of Homotopy and Related Structures 1 (1) (2006), 79-100.

%\bibitem{dicubes2d} Thomas Kahl, Some collapsing operations for 2-dimensional precubical sets, preprint, available at http://w3.math.uminho.pt/$\sim$kahl 

\bibitem{weakmor} T. Kahl, Weak morphisms of higher dimensional automata, arXiv:1303.2003v1 (2013).

%\bibitem{CHomP} T. Kaczynski, K. Mischaikow, M. Mrozek, \emph{Computational Homology}, Applied Mathematical Sciences 157, Springer-Verlag (2004).

%\bibitem{MannaPnueli} Z. Manna, A. Pnueli, \emph{The Temporal Logic of Reactive and Concurrent Systems: Specification}, Springer-Verlag (1992).

%\bibitem{Mazurkiewicz} A. Mazurkiewicz, Trace theory, in: W. Brauer, W. Reisig, G. Rozenberg (eds.), Petri Nets: Applications and Relationships to Other Models of Concurrency, Lecture Notes in Computer Science 255,  Springer-Verlag (1987), 279-324.

%\bibitem{Pin} J.-É. Pin, \emph{Mathematical Foundations of Automata Theory}, Version of January 30, 2012.

%\bibitem{Quillen} D. Quillen, \emph{Homotopical Algebra}, LNM 43, Springer-Verlag (1967).

%\bibitem{RaussenInvariants} M. Rau\ss en, Invariants of Directed Spaces, \emph{Applied Categorical Structures} 15 (2007), 355-386.

%\bibitem{RaussenTraceCub} M. Rau\ss en, Trace spaces in a pre-cubical complex, \emph{Topology and its Applications} 156 (2009), 1718-1728.

%\bibitem{Spanier} E.H. Spanier, \emph{Algebraic Topology}, Springer-Verlag (1966).

%\bibitem{NC1} A. Str\o m, Note on cofibrations, \emph{Math. Scand.} 19 (1966), 11-14.

\bibitem{Sakarovitch} J. Sakarovitch, Elements of Automata Theory (Encyclopedia of Mathematics and Its Applications), Cambridge University Press (2009).

%\bibitem{Strom} A. Str\o m, The homotopy category is a homotopy category, \emph{Arch. Math.} 23 (1972), 435 - 441.



%\bibitem{Whitehead} G. Whitehead, \emph{Elements of Homotopy Theory}, GTM 61, Springer-Verlag (1978).

%\bibitem{WinskelNielsen} G. Winskel, M. Nielsen, Models for concurrency, Handbook of logic in computer science (vol. 4): semantic modelling, Oxford University Press (1995), 1-148.






\end{thebibliography}
\end{document}